\documentclass[12pt,reqno]{amsart}
\usepackage[utf8]{inputenc}
\usepackage[T1]{fontenc}

\title[Analogy between optimal transport and minimal entropy]{About the analogy between optimal transport and minimal entropy}
\date{\today}

\author{Ivan Gentil}
\address{Ivan Gentil. Institut Camille Jordan, Umr Cnrs 5208. Université Claude Bernard. Lyon, France}
\email{gentil@math.univ-lyon1.fr}
 
\author{Christian Léonard}
\address{Christian Léonard. Modal-X. Université Paris Ouest. Nanterre, France}
\email{christian.leonard@u-paris10.fr}

\author{Luigia Ripani}
\address{Luigia Ripani. Institut Camille Jordan, Umr Cnrs 5208. Université Claude Bernard. Lyon, France}
\email{ripani@math.univ-lyon1.fr}

\usepackage{amssymb, amsmath, amsfonts, latexsym}
\usepackage{enumerate,graphicx}

\RequirePackage[colorlinks,linkcolor=blue,citecolor=blue,urlcolor=blue]{hyperref}
\usepackage[english]{babel}

\usepackage{color}
\usepackage{pst-fill,pst-grad,pst-plot,pst-eucl,pstricks-add,pst-node}

\newtheorem{theorem}[equation]{Theorem}
\newtheorem{lemma}[equation]{Lemma}
\newtheorem{proposition}[equation]{Proposition}
\newtheorem{corollary}[equation]{Corollary}

\newtheorem{definition}[equation]{Definition}

\theoremstyle{remark}
\newtheorem{remark}[equation]{Remark}
\newtheorem{remarks}[equation]{Remarks}

\numberwithin{equation}{section}

\newcommand\pf{_{\#}}

\newcommand{\Leb}{\mathrm{Leb}}
\newcommand{\as}{ \textrm{-}\mathrm{a.s.}}

\newcommand{\ud}{\frac{1}{2}}

\newcommand{\R}{\mathbb{R}}
\newcommand\OO{\Omega}
\newcommand\ii{{[0,1]}}
\newcommand\PO{\mathcal{P}(\OO)}
\newcommand\MO{\mathcal{M}(\OO)}
\newcommand{\N}{ \mathcal{N}}

\newcommand\IO{\int_\OO}

\newcommand\ZZ{\mathbb{R}^n}
\newcommand\ZZZ{\ZZ\times\ZZ}
\newcommand\PZ{\mathcal{P}(\ZZ)}
\newcommand\PdZ{\mathcal{P}_2(\ZZ)}

\newcommand\IZ{\int_{\ZZ}}
\newcommand\IZZ{\int_{\ZZZ}}
\newcommand{\IiZ}{\int _{ \ZZ\times\ii}}

\newcommand\XX{\mathbb{R}^n}

\newcommand\MX{\mathcal{M}(\XX)}
\newcommand\IX{\int_{\XX}}

\newcommand{\Rm}{R_{ \mu_{0}}}
\newcommand{\AR}{\mathcal{A}^R}

\begin{document}

\subjclass[2010]{}
\thanks{Authors partially supported by the French ANR projects: GeMeCoD and  ANR-12-BS01-0019 STAB. The third author is supported by the LABEX MILYON (ANR-10-LABX-0070) of Université de Lyon, within the program "Investissements d'Avenir" (ANR-11-IDEX-0007) operated by the French National Research Agency (ANR).}


\begin{abstract}
We describe some analogy between optimal transport and the Schrödinger problem where the  transport cost is replaced by an entropic cost  with a reference path measure. A dual Kantorovich type formulation and  a Benamou-Brenier type representation formula of the entropic cost are derived, as well as contraction inequalities with respect to the entropic cost. This analogy is also illustrated with some numerical examples  where the reference path measure is given by  the Brownian motion or the Ornstein-Uhlenbeck process.

Our point of view is measure theoretical, rather than based on stochastic optimal control, and the relative entropy with respect to path measures plays a prominent role. 
 \end{abstract}

\maketitle 

\noindent
{\sc Résumé.} Nous décrivons des analogies entre le transport optimal et le problème de Schrödinger lorsque le coût du transport est remplacé par un coût entropique avec une mesure de référence sur les trajectoires. Une formule duale de Kantorovich, une formulation de type Benamou-Brenier du coût entropique sont démontrées,  ainsi que des inégalités de contraction par rapport au coût entropique. Cette analogie est aussi illustrée par des exemples numériques où la mesure de référence sur les trajectoires est donnée par le mouvement  Brownien ou bien le processus d'Ornstein-Uhlenbeck.    \\
Notre approche s'appuie sur la théorie de la mesure, plutôt que 
sur le contrôle optimal stochastique, et l'entropie relative joue un rôle fondamental.


\section{Introduction}

In this article,   some analogy between  optimal transport and the Schrödinger problem is investigated.  A    Kantorovich type dual equality,  a Benamou-Brenier type representation  of the entropic cost and contraction inequalities with respect to the entropic cost are derived  when the  transport cost is replaced by an entropic one. This analogy is also  illustrated   with some  numerical examples.
\\
Our point of view is measure theoretical rather than based on stochastic optimal control as is done in the recent literature; the relative entropy with respect to path measures plays a prominent role.

Before explaining  the Schrödinger problem which is associated to an entropy minimization, we first introduce the Wasserstein quadratic transport cost $W_2^2$ and its main properties. For simplicity, our results are stated in $\ZZ$ rather than in a general Polish space. Let us note that properties of the quadratic transport cost can be found in the monumental work by C. Villani~\cite{villani1,villani2}. In particular its square root $W_2$ is a (pseudo-)distance on the space of probability measures which is called  Wasserstein distance. It has been intensively studied and has many interesting applications. For instance it is an efficient tool  for proving convergence to equilibrium of evolution equations, concentration inequalities for measures or stochastic processes and it allows to define curvature in metric measure spaces, see the textbook \cite{villani2} for these applications and more.

\subsection*{The Wasserstein quadratic cost $W_2^2$ and the Monge-Kantorovich problem}
Let $\PZ$ be the set of all probability measures on $\ZZ$. We denote its subset of probability measures with a second moment by $\PdZ=\{\mu\in\PZ;\int |x|^2\,\mu(dx)<\infty\}$. For any $\mu_0,\mu_1\in \PdZ$, the Wasserstein quadratic cost is    
\begin{equation}\label{eq-03}
W_2^2(\mu_0,\mu_1)=\inf_\pi {\IZZ |y-x|^2\, \pi(dxdy)},
\end{equation}
where the infimum is running over all the couplings $ \pi$ of $ \mu_0$ and $ \mu_1$, namely, all the probability measures $\pi$ on $\ZZZ$ with marginals 
$\mu_0$ and $\mu_1$, that is   
for any bounded measurable  functions $\varphi$ and $\psi$ on $\ZZ$,
\begin{equation}\label{eq-04}
\IZZ (\varphi(x)+ \psi(y))\,\pi(dxdy)=\IZ \varphi\, d\mu_0+\IZ \psi\, d\mu_1. 
\end{equation}
In restriction to $\PdZ,$ the pseudo-distance $W_2$ becomes a genuine distance. The {\it Monge-Kantorovich problem} with a quadratic cost function, consists in finding the optimal couplings $\pi$ that minimize~\eqref{eq-03}.

\subsection*{The entropic cost $ \mathcal{A}^R$ and the Schrödinger problem}

Let fix some reference nonnegative measure $R$ on the path space $\Omega= C([0,1], \ZZ)$ and denote $R _{ 01}$ the measure on $\ZZ\times\ZZ$. It describes the joint law of the initial position $X_0$  and the final position $X_1$ of a random process on $\ZZ$ whose law is $R$. This means that $$R _{ 01}=(X_0,X_1)\pf R$$ is the push-forward  of $R$ by the mapping $(X_0,X_1).$   Recall that the push-forward of a measure $ \alpha$ on the space $ \mathsf{A}$ by the measurable mapping $f: \mathsf{A}\to \mathsf{B}$ is defined by 
$$
 f\pf \alpha(db)= \alpha( f ^{ -1}(db)),\qquad db\subset \mathsf{B},
 $$
 in other words, for any positive function $H$, 
 $$
 \int Hd(f\pf \alpha)=\int H(f)d\alpha.
 $$
For  any   probability measures  $\mu_0$ and  $\mu_1$ on $\ZZ$, the entropic cost  $\mathcal A^R(\mu_0,\mu_1)$ of $( \mu_0, \mu_1)$  is defined by
$$
\mathcal{A}^{R}(\mu_0,\mu_1) = \inf_\pi H(\pi|R _{ 01})
$$
where $H(\pi|R _{ 01})=\int _{ \ZZ\times\ZZ}\log( d \pi/ d R _{ 01})\, d \pi$ is the relative entropy of $\pi$ with respect to $R _{ 01}$ and $\pi$   runs through all the couplings of $ \mu_0$ and $ \mu_1$. The {\it Schrödinger problem} consists in finding the unique optimal entropic plan $\pi$ that minimizes the above infimum. 

In this article, we choose $R$ as the reversible Kolmogorov continuous Markov process specified by the generator $\frac{1}{2}(\Delta-\nabla V\cdot\nabla )$ and the initial reversing measure $e ^{ -V(x)}\,dx$.  

\subsection*{Aim of the paper}

Below in this introductory section, we are going to  focus onto four main  features of the quadratic transport cost $W_2^2$. Namely:
\begin{itemize}
\item
the Kantorovich dual formulation of $W_2^2$;
\item
the Benamou-Brenier dynamical formulation of $W_2^2$;
\item
the displacement interpolations, that is the $W_2$-geodesics in  $\PdZ$;
\item the contraction of the heat equation with respect to $W_2^2.$ 
\end{itemize} 
The goal of this article is to recover analogous properties for $ \mathcal{A}^R$  instead of $W_2^2$, by replacing the Monge-Kantorovich problem with the Schrödinger problem. 

\subsection*{Several aspects of  the quadratic Wasserstein cost}

Let us provide some detail about these four properties.  
\subsubsection*{Kantorovich dual formulation of $W_2^2$}

The following duality result was proved by Kantorovich in~\cite{kantorovich}. For any $\mu_0,\mu_1\in \PdZ$, 
\begin{equation}\label{eq-05}
W_2^2(\mu_0,\mu_1)=\sup _{  \psi}\left\{\IZ \psi \,d\mu_1-\IZ Q\psi \,d\mu_0\right\}, 
\end{equation}
where the supremum  runs over all bounded  continuous function $ \psi$ and 
$$
Q\psi(x)=\sup_{y\in\ZZ}\left\{\psi(y)-|x-y|^2\right\}, \quad x\in\ZZ.
$$ 
It is often expressed in the equivalent form, 
\begin{equation}\label{eq-kanto-classique}
W_2^2(\mu_0,\mu_1)=\sup  _{  \varphi}\left\{\IZ \widetilde{Q} \varphi \,d\mu_1-\IZ \varphi \,d\mu_0\right\}, 
\end{equation}
where the supremum  runs over all bounded  continuous function $ \varphi$ and 
$$
\widetilde{Q} \varphi(y)=\inf_{x\in\ZZ}\left\{ \varphi(x)+|y-x|^2\right\},
\quad y\in\ZZ.
$$ 
The map $Q \psi$ is called  the sup-convolution of $ \psi$ and its defining identity is sometimes referred to as the Hopf-Lax formula.

\subsubsection*{Benamou-Brenier formulation of $W^2_2$}

The Wasserstein cost admits  a dynamical formulation: the so-called Benamou-Brenier formulation which was proposed in~\cite{benamou-brenier}.  It states that for any $\mu_0,\mu_1\in \PdZ$, 
\begin{equation}\label{bena}
W_2^2(\mu_0,\mu_1)=\inf _{ (\nu,v)} \IiZ  |v_t|^2\,d \nu_t\, dt,
\end{equation}
where the infimum  runs over all  paths $(\nu_t,v_t)_{t\in[0,1]}$ where   $\nu_t\in\PZ$ and $v_t(x)\in\ZZ$  are such that $\nu_t$ is absolutely continuous with respect to time in the sense of~\cite[Ch.\,1]{ags} for all $0\le t\le 1$, $\nu_0=\mu_0$, $\nu_1=\mu_1$ and 
$$
\partial_t\nu_t+\nabla \cdot(\nu_t v_t)=0,
\quad 0\le t\le 1.
$$
In this equation which is understood in a weak sense, 
 $\nabla\cdot\,$ stands for the standard divergence of a vector field  in $\ZZ$ and $\nu_t$ is identified with its density with respect to Lebesgue measure. This general result is proved in~\cite[Ch.\,8]{ags}. A proof  under the additional  requirement that $\mu_0,\mu_1$ have  compact supports is available 
 in~\cite[Thm.\,8.1]{villani1}.

\subsubsection*{Displacement interpolations}  

The metric space $(\PdZ,W_2)$ is  geodesic. This means that for any probability measure $\mu_0,\mu_1\in \PdZ$, there exists a path $(\mu_t)_{t\in[0,1]}$ in $\PdZ$ such that   for any $s,t\in[0,1]$, 
$$
W_2(\mu_s,\mu_t)=|t-s|W_2(\mu_0,\mu_1). 
$$
Such a  path is  a constant speed geodesic in $(\PdZ,W_2)$, see~\cite[Ch.\,7]{ags}. Moreover when $\nu$ is absolutely continuous with respect to the Lebesgue measure, there exists a  convex function $ \psi$ on $\ZZ$ such that for any $t\in[0,1]$, the geodesic is given by
\begin{equation}\label{mccan}
\mu_t=((1-t) \rm{Id}+t\nabla  \psi)_\#\mu_0. 
\end{equation}
This interpolation is called the McCann displacement interpolation in $(\PdZ,W_2)$, see~\cite[Ch.\,5]{villani1}.

\subsubsection*{Contractions in Wasserstein distance}  

Contraction in Wasserstein distance is a way to define the curvature of the underlying space or of the reference Markov operator. In its general formulation, the von Renesse-Sturm theorem tells that the heat equation in a smooth, complete and   connected Riemannian manifold satisfies a contraction property with respect to the Wasserstein distance if and only if the Ricci curvature is bounded from below, see~\cite{renesse-sturm}. 
In the context of the present article where Kolmogorov semigroups on $\ZZ$ are considered,  two main contraction results will be explained with more details in Section~\ref{sec-contraction}.

\subsection*{Organization  of the paper}  
 The setting of the present work  and notation are introduced in Section~\ref{sec-setting}.  The entropic  cost $ \mathcal{A}^R$ is defined with more detail in Section~\ref{sec-schrodinger} together with the related  notion of entropic interpolation, an analogue of the displacement interpolation.  A dual Kantorovich  type formulation  and  a Benamou-Brenier type formulation of the entropic cost are derived respectively at  Sections~\ref{sec-kantorovich} and~\ref{sec-sbb}. Section~\ref{sec-contraction} is dedicated to the contraction properties of the heat flow with respect to the entropic cost.  In the last Section~\ref{sec-examples}, we give some examples of entropic interpolations  between Gaussian distributions when the reference path measure is given by the Brownian motion or the Ornstein-Uhlenbeck process.

\subsection*{Literature}
 The Benamou-Brenier formulation of  the entropic cost which is stated at Corollary~\ref{res-03}  was proved recently by Chen, Georgiou and Pavon in \cite{CGP14} (in a slightly less general setting)  without any mention to optimal transport (in this respect Corollary~\ref{cor-super} relating the entropic and Wasserstein costs is new). Although our proof of Corollary~\ref{res-03}  is close to their proof, we think it is worth including it in the present article to emphasize the analogies between displacement and entropic interpolations. In addition, we also provide a time asymmetric version of this formulation at Theorem~\ref{teo:teo}.
\\
Both  \cite{MT06} and \cite{CGP14}  share the same stochastic optimal control viewpoint. This differs from the entropic approach of the present paper.
\\
Let us notice that Theorem~\ref{thm-contr} is a new result: it provides contraction inequalities with respect to the entropic cost. Moreover, examples and comparison  proposed at the end of the paper, with respect two different kernels (Gaussian and Ornstein-Uhlenbeck) are new.

\section{The reference path measure}
\label{sec-setting}

 We make precise the reference path measure $R$ to which the entropic cost $ \mathcal{A}^R$ is associated. Although more general reversible path measures $R$ would be all right to define a well-suited entropic cost, we prefer  to consider the specific class of Kolmogorov Markov measures. This choice is motivated by the fact that, as presented in \cite{leonard12}, the Monge Kantorovich problem is the limit of a sequence of entropy minimization problems, when a proper fluctuation parameter tends to zero. The Kolmogorov Markov measures, as reference measures in the Schrödinger problem, admit as a limit case the Monge Kantorovich problem with quadratic cost function, namely the Wasserstein distance.

\subsection*{Notation}

For any  measurable set $Y$, we denote respectively by $\mathcal P(Y)$ and $ \mathcal{M}(Y)$  the set of all the probability measures and all positive $ \sigma$-finite measures on $Y$. The \emph{relative entropy} of a probability measure $p \in \mathcal P(Y)$ with respect to a positive  measure $r \in  \mathcal{M}(Y)$ is loosely defined by
\begin{equation*}
H(p|r) :=  \left\{ \begin{array}{ll} 
\int_{Y} \log (dp/dr) dp \in (-\infty, \infty], & \textrm{if }p\ll r,\\
\infty,& \textrm{otherwise.} 
\end{array}\right.
\end{equation*}
For some assumptions on the reference measure $r$ that guarantee the above integral to be meaningful and bounded from below, see after  the regularity hypothesis (Reg2) at page~\pageref{page-reg}. For a rigorous definition and some properties of the relative entropy with respect to an unbounded measure see~\cite{Leo12b}. The state space $\R^n$ is equipped with its Borel $\sigma$-field and  the path space $\Omega$ with the canonical $\sigma$-field $ \sigma(X_t; 0\le t\le 1)$  generated by the canonical process 
$$
X_{t}(\omega) := \omega_{t} \in \R^n, \quad \omega=(\omega_{s})_{0\leq s \leq 1}\in\OO,\ 0\le t\le 1.
$$
For any {path measure} $Q \in \MO$ and any $0\le t\le 1$, we denote
$$
Q_{t}(\cdot) := Q(X_t\in \cdot) = (X_{t})\pf Q \in \mathcal M(\R^n),
$$
the push-forward of $Q$ by $X_t$. When $Q$ is a probability measure, $Q_t$ is  the law of the random location $X_{t}$ at time $t$ when the law of the whole trajectory is $Q$.

\subsection*{The Kolmogorov Markov measure $R$ and its semigroup}
\label{sec-kolmo}

Most of our results can be stated in the general setting  of a Polish state space. 
For the sake of simplicity, the setting of the present paper is particularized. The state space is $\ZZ$  and the reference path measure $R$ is the Markov path measure associated with the generator
\begin{equation}\label{eq-11}
\ud ( \Delta-\nabla V\cdot\nabla)
\end{equation}
and the corresponding reversible measure 
$$
m=e ^{ -V}\,\Leb
$$ 
as its initial measure, where $\Leb$ is the Lebesgue measure. It is assumed that the potential   $V$ is a $ \mathcal{C}^2$  function on $\ZZ$ such that the martingale problem associated with the generator~\eqref{eq-11} on the domain $\mathcal{C}^2 $ and the initial measure $m$ admits a unique solution $R\in\MO$. This is the case for instance when the following hypothesis are satisfied. 

\subsubsection*{Existence hypothesis (Exi)}
There exists some constant $c>0$ such that one of the following assumptions holds true:
\begin{enumerate}[(i)]
\item
$\lim _{ |x|\to \infty}V(x)= + \infty$ and $ \inf \{|\nabla V|^2- \Delta V/2\}>- \infty,$ or
\item
$-x\cdot \nabla V(x)\le c(1+|x|^2),$ for all $x\in\ZZ.$
\end{enumerate}
\medskip

See \cite[Thm.\,2.2.19]{Roy99} for the existence result under the assumptions (i) or (ii).
For any initial condition $X_0=x\in\ZZ,$ the path measure $R_x:=R(\cdot\mid X_0=x)\in\PO$ is the law of the weak solution of  the stochastic differential equation
\begin{equation}\label{eq-11b}
dX_t=-\nabla V(X_t)/2\ dt+ dW_x(t), \quad 0\le t\le 1
\end{equation}
where $W_x$ is an $R_x$-Brownian motion.
The Kolmogorov Markov measure is $$R(\cdot)=\int _{ \ZZ}R_x(\cdot)\,m(dx)\in\mathcal M( \Omega).$$
Recall that $m= e ^{ -V}\,\Leb$ is not necessary a probability measure. 
The Markov semigroup associated to $R$ is defined for any bounded measurable function $f:\ZZ\mapsto \mathbb{R}$ and  any $t\geq0,$ by
$$
T_tf(x)=E _{ R_x}f(X_t),\quad x\in\ZZ.
$$
It is reversible  with reversing measure $m$ as defined in~\cite{bgl}. 
\\
\subsubsection*{Regularity hypothesis (Reg1)}
We also assume for simplicity that $(T_t)_{t\geq0}$ admits for any $t>0$, a  {\it density kernel} with respect to $m$,  a probability density $p_t(x,y)$ such that
\begin{equation}\label{dens}
T_tf(x) = \int_{\ZZ} f(y) p_t(x,y)\, m(dy).
\end{equation}
For instance, when $V(x)=|x|^2/2$, then $R$ is the path measure associated to the Ornstein-Uhlenbeck process with the Gaussian measure as its reversing measure. When $V=0$, we recover the Brownian motion with  Lebesgue measure   as its reversing measure. Examples of Kolmogorov semigroups admitting a density kernel can be found for instance in~\cite[Ch. 3]{stroock} or~\cite[Cor.\,4.2]{bbgm}. This semigroup is fixed once for all. 

\subsection*{Properties of the path measure $R$}
The measure $R$ is our {\it reference} path measure and it satisfies the following properties.
\begin{enumerate}[(a)]
\item It is Markov, that is for any $t\in[0,1]$, 
$
R(X_{[t,1]}\in\cdot|X_{[0,t]})=R(X_{[t,1]}\in\cdot|X_{t}).
$
See~\cite{Leo12b} for the definition of the conditional law for unbounded measures since $R$ is not necessary a probability measure. 
\item It is reversible. This means that for all $0\le T\le 1$, the restriction $R _{ [0,T]}$ of $R$ to the sigma-field $ \sigma(X _{ [0,T]})$ generated by $X _{ [0,T]}=(X_t) _{ 0\le t\le T},$ is invariant with respect to time-reversal, that is $[(X _{ T-t}) _{ 0\le t\le T}]\pf R _{ [0,T]}=R _{ [0,T]}.$ 
\\
Any reversible measure $R$ is stationary, i.e.\ $R_t=m,$ for all $0\le t\le 1$ for some $m\in\MX.$ This measure $m$ is called the reversing measure of $R$ and  is often interpreted as an equilibrium of the dynamics specified by the kernel $(R_x; x\in\XX)$. One says for short that $R$ is $m$-reversible.
\end{enumerate}

\section{Entropic cost and entropic interpolations}
\label{sec-schrodinger}
We define  the Schrödinger problem, the entropic cost and the entropic interpolation which are respectively the analogues of the Monge-Kantorovich problem, the Wasserstein cost and the displacement interpolation that were briefly described in the introduction.


Let us state the definition of the entropic cost associated with $R.$

\begin{definition}[Entropic cost]  Consider the projection $$R_{01} := (X_{0},X_{1})_{\pf}R \in \mathcal M(\mathbb R^n\times\mathbb R^n)$$ of $R$ onto the endpoint space $\mathbb R^n\times\mathbb R^n$.
For any  $\mu_0,\mu_1\in \PZ$,  
 \begin{equation*}
\AR(\mu_0,\mu_1) = \inf \{H(\pi|R_{01}) ; \  \pi \in \mathcal P(\mathbb R^n\times\mathbb R^n): \pi_{0}=\mu_{0}, \pi_{1}=\mu_{1}\}\in (- \infty, \infty]
\end{equation*}
is the $R$-entropic cost of $( \mu_0, \mu_1)$. 
\end{definition}

This definition  is related to a static Schrödinger problem. It also admits a dynamical formulation.
\begin{definition}[Dynamical formulation of the Schrödinger problem]
The Schrödinger problem associated to $R, \mu_0$ and $ \mu_1$ consists in  finding the minimizer $\hat P$ of the relative entropy $H(\cdot|R)$ among all the probability  path measures $P\in\PO$ with prescribed initial and final marginals $P_0= \mu_0$ and $P_1= \mu_1$,
\begin{equation}\label{eq-07}
H(\hat P|R)=\min\{H(P|R),\,\,P\in\PO, P_0= \mu_0, P_1= \mu_1\}.
\end{equation}
\end{definition}
 It is easily seen that  its  minimal value is the {entropic cost},
 \begin{equation}\label{eq-14}
\AR(\mu_0,\mu_1) = \inf \{  H(P|R); \  P\in\PO: P_0= \mu_0, P_1= \mu_1\}\in (- \infty, \infty],
\end{equation}
 see for instance \cite[Lemma\,2.4]{leonard14}.

In the rest of this work, the entropic cost will always be associated with the fixed reference measure $R \in \MO$, therefore without ambiguity we  drop the index and denote $\mathcal A^R = \mathcal A$.

\medskip

\begin{remarks}
\begin{enumerate}[(1)]
\item
First of all, when $R$ is not a probability measure, the relative entropy might take some negative values and even the value $- \infty$. However, because of the decrease of information by push-forward mappings, we have\begin{equation*}
H(P|R)\ge \max(H( \mu_0|m),H( \mu_1|m)),
\end{equation*}
see   \cite[Thm.\,2.4]{Leo12b} for instance.
Hence $H(P|R)$ is well defined in $(- \infty, \infty]$ whenever $H( \mu_0|m)> -  \infty$ or $H( \mu_1|m)> - \infty.$ This will always be assumed.

\item
Even the nonnegative quantity $ \mathcal{A}( \mu_0, \mu_1)-\max(H( \mu_0|m),H( \mu_1|m))\ge 0$  cannot be the 
 square of a distance such as the Wasserstein cost   $W_2^2$. 
As a matter of fact, considering the special situation where $ \mu_0= \mu_1= \mu,$ we have $ \mathcal A( \mu, \mu)\geq H( \mu|m)>0$  as soon as $ \mu$ differs from $m$. This is a consequence of Theorem~\ref{teo:teo} below.
\item
A good news about $ \mathcal A$ is that since $R$ is reversible, it is symmetric:  $ \mathcal A( \mu, \nu) = \mathcal A( \nu, \mu)$. To see this, let us denote  $X^*_t=X _{ 1-t}, 0\le t\le 1,$ and $Q^*:=(X^*)\pf Q$ the time reversal of any $Q\in\MO.$ As $X^*$  is one-one, we have $H(P|R)=H(P^*|R^*)$ and since we assume that $R^*=R,$ we see that 
\begin{equation}\label{eq-19}
H(P|R)=H(P^*|R),\quad \forall P\in\PO.
\end{equation}
 Hence, if $P$ solves~\eqref{eq-07} with $ (\mu_0, \mu_1)= (\mu, \nu),$ then ${X^*}\pf P$ solves~\eqref{eq-07} with $( \mu_0, \mu_1)=( \nu, \mu)$ and these Schrödinger problems share the same value. 
  \end{enumerate}
 \end{remarks}


\subsection*{Existence of a minimizer. Entropic interpolation}

We recall some general results  from  \cite[Thm.\,2.12]{leonard14} about the solution of the dynamical Schrödinger problem~\eqref{eq-07}. Let us denote by $p(x,y)$ the probability density introduced at~\eqref{dens}, at time $t=1$, so that 
\begin{equation*}
R _{ 01}(dxdy)=m(dx)p(x,y)m(dy).
\end{equation*}
In order for~\eqref{eq-07} to admit a unique solution, it is enough that it satisfies the following hypothesis:
\subsubsection*{Regularity hypothesis (Reg2)}\label{page-reg}
\begin{enumerate}[(i)]
\item $p(x,y)  \geq e^{-A(x)-A(y)}$ for some nonnegative measurable function $A$ on $\mathbb{R}^n;$
\item $\int_{\ZZ \times \ZZ} e^{-B(x)-B(y)}p(x,y)\,m(dx) m(dy)< \infty$ for some nonnegative measurable function $B$ on $\mathbb{R}^n;$
\item $\IZ (A+B)\, d\mu_0, \IZ (A+B)\, d\mu_1<\infty$ where $A$ appears at (i) and $B$ appears at (ii);
\item $- \infty<H(\mu_0|m), H(\mu_1|m)<\infty$;
\end{enumerate}
Assumptions (ii)-(iii) are useful to define rigorously $H(\mu_0|m)$ and  $H(\mu_1|m)$.  Under these assumptions the entropic cost  $\mathcal{A}(\mu_0,\mu_1)$ is finite and the minimizer $P$ of the Schrödinger problem~\eqref{eq-07} is characterized, in addition to the marginal constraints $P_0= \mu_0, P_1= \mu_1,$ by the product  formula
\begin{equation}\label{eq-06}
P=f_0(X_0)g_1(X_1)\,R
\end{equation}
for some measurable functions $f_0$ and $g_1$ on $\ZZ.$ The uniqueness of the solution is a direct consequence of the fact that~\eqref{eq-07} is a {\it strictly} convex minimization problem.

\begin{definition}[Entropic interpolation]
The {\it $R$-entropic interpolation} between $ \mu_0$ and $ \mu_1$ is defined as the marginal flow of the minimizer $P$ of~\eqref{eq-07}, that is $\mu_t :=P_t \in \PZ, 0\le t \le 1$.
\end{definition}

\begin{proposition}
Under the hypotheses (Exi), (Reg1) and (Reg2), the $R$-entropic interpolation between $ \mu_0$ and $ \mu_1$ is characterized by
\begin{equation}\label{entropic}
\mu_t = e^{\varphi_t+\psi_t}\,{m}, \quad 0\le t\le 1,
\end{equation}
where  
\begin{equation}\label{eq-16}
\varphi_t = \log T_t f_0,\quad  \psi_t = \log T_{1-t} g_1 ,\quad 0\le t\le 1,
\end{equation}
and the measurable functions  $f_0, g_1$ solve the following system
\begin{equation}\label{eq-10}
\frac{d \mu_0}{dm} = f_0 T_1 g_1,\quad \frac{d \mu_1}{dm}= g_1 T_1 f_0.
\end{equation}
\end{proposition}
The system~\eqref{eq-10} is often called the Schrödinger system. It simply expresses the marginal constraints. Its solutions $(f_0,g_1)$ are precisely the functions that appear in the identity~\eqref{eq-06}. Actually it is difficult or impossible to solve explicitly the system~\eqref{eq-10}. However, in Section~\ref{sec-examples},  we will see some particular examples  in the Gaussian setting where the system admits an explicit solution. Some numerical  algorithms have been proposed recently in~\cite{algo-dauphine}.

In our setting where $R$ is the Kolmogorov path measure defined at~\eqref{eq-11},   the entropic interpolation  $\mu_t$ admits a density $ \mu_t(z):=d \mu_t/dz$ with respect to the Lebesgue measure. It is important to notice that, contrary to the McCann interpolation, the   $(t,x)\mapsto \mu_t(x)$ is smooth on $(t,x)\in]0,1[\times \R$ and solves the transport equation
\begin{equation}\label{eq-12}
\partial_t \mu_t + \nabla \cdot (\mu_t\, v^{cu}(t, \mu_t))=0
\end{equation}
with the initial condition $ \mu_0$ and  where   $v^{cu}(t, \mu_t,z)=\nabla \psi_t(z)-\nabla V(z)/2+\nabla \log \mu_t(z)/2  $  refers to the {\it current velocity} introduced by Nelson in~\cite[Chap.\,11]{nelson} (this will be recalled at Section~\ref{sec-sbb}).   The current velocity is a smooth function and the ordinary differential equation 
$$
\dot x_t(x)=v^{cu}(t,x_t(x)), \quad x_0=x
$$
admits a unique solution for any initial position $x$, the solution of the continuity equation~\eqref{eq-12} admits the following push-forward expression:
\begin{equation}\label{eq-20}
\mu_t= (x_t)\pf \mu_0, \quad 0 \leq t \leq 1,
\end{equation}
in analogy with the displacement interpolation given at~\eqref{mccan}.

\begin{remark}[From the entropic cost to the Wasserstein cost]
\label{rem-entropic-wasserstein}
The Wasserstein distance is a limit case of the entropic cost. We shall  use this result to compare contraction properties in  Section~\ref{sec-contraction} and also to illustrate the examples in Section~\ref{sec-examples}.
\\
Let us consider the following dilatation  in time with ratio $\varepsilon>0$ of the reference path measure $R$: $R ^{  \epsilon}:=(X^ \epsilon)\pf R$ where $X^ \epsilon(t):=X _{  \epsilon t},$ $0\le t\le 1.$  It is shown in  \cite{leonard12} that  some renormalization of the  entropic cost $ \mathcal{A} ^{ R ^ \epsilon}$ converges  to the Wasserstein distance when $\varepsilon$ goes to 0. Namely,
\begin{equation}\label{eq-gamma-convergence}
\lim_{\varepsilon \to 0}\varepsilon \mathcal A^{R^ \epsilon}(\mu_0, \mu_1) = W_2^2(\mu_0, \mu_1)/2.
\end{equation} 
Even better, when   $ \mu_0$ and $\mu_1$ are absolutely continuous,  the entropic interpolation $( \mu ^{ R^ \epsilon}_t) _{ 0\le t\le 1}$ between $ \mu_0$ and $ \mu_1$ converges as $ \epsilon$ tends to zero towards the McCann displacement interpolation $( \mu_t) _{ 0\le t\le 1},$ see~\eqref{mccan}.
\end{remark}

\section{Kantorovich dual equality for the entropic cost}
\label{sec-kantorovich}

We derive the analogue of the Kantorovich dual equality~\eqref{eq-05} when the Wasserstein cost is replaced by the entropic cost.

\begin{theorem}[Kantorovich dual equality for the entropic cost]
\label{res-02}
Let $V,\mu_0$ and $\mu_1$ be such that the hypothesis (Exi), (Reg1) and (Reg2) stated in Section~\ref{sec-schrodinger} are satisfied. We have 
\begin{equation*}
\mathcal A(\mu_0, \mu_1)=H( \mu_0|m)
+\sup\left\{\IZ \psi\, d \mu_1-\IZ \mathcal{Q}^{ R} \psi\, d \mu_0;\  \psi\in C_{b}(\ZZ)\right\}
\end{equation*}
where for every $\psi\in C_{b}(\ZZ)$,
$	
\mathcal{Q}^{ R} \psi(x):=\log E _{ R^x} e ^{ \psi(X_1)}=\log T_1(e^\psi)(x),\ x\in\ZZ.
$	

\end{theorem}

This result was obtained by Mikami and Thieullen in \cite{MT06} with an alternate statement and a  different proof. The present proof is based on an abstract dual equality which is stated below at Lemma~\ref{res-01}. Let us first describe the setting of this lemma.

Let $\mathbb U$ be a vector space and $\Phi:\mathbb U \to (- \infty, \infty]$ be an extended real valued function on $\mathbb U.$  Its convex conjugate $\Phi^*$ on the algebraic dual space  $\mathbb U^*$  of $\mathbb U$ is defined  by 
$$
\Phi^*(\ell):=\sup _{ u\in \mathbb U}\left\{ \left\langle \ell,u \right\rangle _{\mathbb U^*,\mathbb U} -\Phi(u)\right\} \in[- \infty, \infty],\qquad \ell\in \mathbb U^*.
$$
We consider a linear map $A:\mathbb U^*\to \mathbb V^*$ defined on $\mathbb U^*$ with values in the algebraic dual space $\mathbb V^*$ of some   vector space $\mathbb V.$ 

\begin{lemma}[Abstract dual equality]\label{res-01}
We assume that:
\begin{enumerate}[(a)]
\item
$\Phi$ is a convex  lower  $\sigma(\mathbb U,\mathbb U^*)$-semicontinuous function and there is some $\ell_o\in \mathbb U^*$ such that for all $u\in \mathbb U,$ $ \Phi(u)\ge \Phi(0)+ \left\langle \ell_o,u \right\rangle _{\mathbb U^*,\mathbb U}$ ;
\item
$\Phi^*$ has $ \sigma(\mathbb U^*,\mathbb U)$-compact level sets: $ \left\{ \ell\in \mathbb U^*: \Phi^*(\ell)\le a\right\} ,$ $a\in\R$;
\item

The algebraic adjoint $A^\dagger $ of $A$ satisfies $A^\dagger \mathbb V\subset \mathbb U.$
\end{enumerate}
Then,  the dual equality 
\begin{equation}\label{eq-02}
\inf \left\{ \Phi^*(\ell); \ell\in \mathbb U^*, A\ell=v^*\right\} 
= \sup _{ v\in \mathbb V} \left\{ \left\langle v,v^* \right\rangle _{ \mathbb V,\mathbb V^*}- \Phi(A^\dagger v)\right\} \in (- \infty, \infty]
\end{equation}
holds true for any $v^*\in \mathbb V^*.$
\end{lemma}

\begin{proof}[Proof of Lemma~\ref{res-01}]
 In the special case where $ \Phi(0)=0$ and $\ell_o=0$, this result is~\cite[Thm.\,2.3]{Leo01c}.  Considering $ \Psi(u):= \Phi(u)-[ \Phi(0)+ \left\langle \ell_o,u \right\rangle ],$ $u\in \mathbb U$, we see that $\inf\Psi=\Psi(0)=0$,  $\Psi$ is a convex  lower  $\sigma(\mathbb U,\mathbb U^*)$-semicontinuous function and $ \Psi^*(\ell)= \Phi^*(\ell_o+\ell)+ \Phi(0),$ $\ell\in \mathbb U^*,$ has $ \sigma(\mathbb U^*,\mathbb U)$-compact level sets. As $ \Psi^*$ satisfies the assumptions of~\cite[Thm.\,2.3]{Leo01c}, we have the dual equality
\begin{equation*}
\inf \left\{ \Psi^*(\ell); \ell\in \mathbb U^*, A\ell=v^*-A\ell_o\right\} 
= \sup _{ v\in \mathbb V} \left\{ \left\langle v,v^*-A\ell_o \right\rangle _{ \mathbb V,\mathbb V^*}- \Psi(A^\dagger v)\right\} \in[0, \infty]
\end{equation*}
which is~\eqref{eq-02}.
\end{proof}
\medskip

\begin{proof}[Proof of Theorem~\ref{res-02}]
Let us denote 
\begin{equation*}
\Rm(\cdot):=\IZ R_x(\cdot)\, \mu_0(dx)\in\PO.
\end{equation*}
$\Rm$ is a measure on paths, its initial marginal, as a probability measure in $\R^n$, is $R_{\mu_0,0}= \mu_0$. When $ \mu_0=m,$ we have   $R_m=R$.    If $U$ is a functional on paths, 
\begin{multline*}
E_{\Rm}(U)=\int E_R(U|X_0=x)\mu_0(dx)=\int E_R(U|X_0=x)\frac{d\mu_0}{dR_0}(x)R_0(dx)\\
=E_R\big(E_R(U\frac{d\mu_0}{dm}(X_0)|X_0)\big)=E_R\big(U\frac{d\mu_0}{dm}(X_0)\big).
\end{multline*}
So 
$\displaystyle{
\Rm(\cdot)=  \frac{d \mu_0}{dm}(X_0) \,R(\cdot),}$
we see that for any $P\in\PO$ such that $P_0= \mu_0,$ 
\begin{equation}\label{eq-09}
  H(P|R)=  H( \mu_0|m)+  H(P|\Rm).
\end{equation}
Consequently, the minimizer of~\eqref{eq-07} is also the minimizer of 
\begin{equation}\label{eq-08}
H(P|\Rm)\to \mathrm{min};\quad  P\in\PO: P_0= \mu_0, P_1= \mu_1
\end{equation}
and 
\begin{equation}\label{eq-13}
\mathcal A(\mu_0,\mu_1)=H( \mu_0|m)+\inf \{H(P|\Rm),\,\, P\in\PO: P_0= \mu_0, P_1= \mu_1
\}.
\end{equation} 
Therefore, all we have to prove is 
\begin{multline*}
\inf \{H(P|\Rm),\,\, P\in\PO: P_0= \mu_0, P_1= \mu_1
\}=\\
\sup\left\{\IZ \psi\, d \mu_1-\IZ \mathcal{Q}^{ R} \psi\, d \mu_0,\  \psi\in C_{b}(\ZZ)\right\}.
\end{multline*}
This  is an application of  Lemma~\ref{res-01} with $\mathbb U=C_b(\OO)$, $\mathbb V=C_b(\ZZ)$ and
\begin{eqnarray*}
\Phi(u)&=&\int _{ \ZZ}\log \left( \IO e^u\,dR^x \right) \, \mu_0(dx),
\quad u\in C_b(\OO),\\
A^\dagger  \psi &=&\psi (X_1)\in C_b(\OO),
\quad \psi \in C_b(\ZZ).
\end{eqnarray*}
Let  $C_b(\OO)'$ be the topological dual space  of $(C_b(\OO),\|\cdot\|)$ equipped with the uniform norm $\| u\|:=\sup _\OO| u|.$
It is shown at \cite[Lemma~4.2]{leonard12} that for any $\ell\in C_b(\OO)',$
\begin{equation}\label{eq-01}
 \Phi^*(\ell) =
 \left\{
 \begin{array}{ll}
  H(\ell|\Rm),& \textrm{if }\ell\in \mathcal P(\Omega) \textrm{ and }(X_0)\pf\ell=\mu_{0}\\
 + \infty,& \textrm{otherwise}
 \end{array}\right.
\end{equation}
But according to \cite[Lemma 2.1]{Leo01a}, the effective domain $\left\{\ell\in C_b(\OO)^*: \Phi^*(\ell)< \infty\right\}$ of $\Phi^*$ is a subset of $C_b(\OO)'. $ Hence, for any $\ell$ in the algebraic dual $C_b(\OO)^*$ of $C_b(\OO),$ $\Phi^*(\ell)$ is given by~\eqref{eq-01}.

The assumption (c) of Lemma~\ref{res-01} on $A^\dagger$ is obviously satisfied.
Let us show  that $ \Phi$ and $ \Phi^*$ satisfy the assumptions (a) and (b). 
\\ 
Let us start with (a). It  is a standard result of the large deviation theory that $u\mapsto \log\IO e^u\,dR^x  $ is convex (a consequence of Hölder's inequality). It follows that $ \Phi$ is also convex. As $ \Phi$ is upper bounded on a neighborhood of $0$ in $(C_b(\OO),\|\cdot\|):$
\begin{equation}
\label{eq-maximum-phi}
\sup _{ u\in C_b(\OO), \|u\|\le 1} \Phi(u)\le 1 < \infty
\end{equation}
(note that $\Phi$ is increasing and $\Phi(1)=1$)  and its effective domain is the whole space $\mathbb U=C_b(\OO)$, it is $\|\cdot\|$-continuous everywhere. Since $ \Phi$ is convex, it is  also lower $ \sigma(C_b(\OO),C_b(\OO)')$-semicontinuous and a fortiori lower $ \sigma(C_b(\OO),C_b(\OO)^*)$-semicontinuous. 
Finally, a direct calculation shows that $\ell_o=R _{  \mu_0}$ is 	a subgradient of $ \Phi$ at 0. This completes the verification of (a).

The assumption  (b) is also satisfied because the upper bound~\eqref{eq-maximum-phi} implies that the level sets of $ \Phi^*$ are $ \sigma(C_b(\OO)^*,C_b(\OO))$-compact, see \cite[Cor.\,2.2]{Leo01a}. So far, we have shown that the assumptions of Lemma~\ref{res-01} are satisfied.

It remains to show that    $A\ell=v^*$ corresponds to the final marginal constraint.  Since $\{ \Phi^*< \infty\}$ consists of probability measures, it is enough to specify the action of $A$ on the vector subspace $ \mathcal{M}_b(\OO)\subset C_b(\OO)^*$ of all bounded  measures on $\OO$. For any $Q\in \mathcal{M}_b(\OO)$ and any $\psi \in C_b(\OO),$ we have 
$$ 
\left\langle \psi ,AQ \right\rangle _{  C_b(\ZZ),C_b(\ZZ)^*}=\left\langle A^\dagger \psi ,Q \right\rangle _{ C_b(\OO),C_b(\OO)^*}=\IO \psi (X_1)\,dQ=\int _{ \ZZ}\psi \,dQ_1.
$$ 
This means that for any $Q\in\mathcal{M}_b(\OO),$ $AQ=Q_1\in \mathcal{M}_b(\ZZ).$ 

With these considerations, choosing $v^*= \mu_1\in \mathcal{P}(\ZZ)$ in~\eqref{eq-02} leads us to 
\begin{alignat*}{1}
\inf \big\{ H(Q|R _{ \mu_0})&;Q\in\PO: Q_0= \mu_0, Q_1= \mu_1\big\} \\
&= \sup _{ \psi \in C_b(\ZZ)} \left\{ \int _{ \ZZ}\psi \, d \mu_1-\int _{ \ZZ}\log \left\langle e ^{ \psi (X_1)},R^x \right\rangle \, \mu_0(dx)\right\} 
\end{alignat*}
which is the desired identity.
\end{proof}

\begin{remark} 
Alternatively, considering $R^y:=R^{\mu_1}(\cdot\mid X_1=y)$, for $m$-almost all $x\in\ZZ$ and
\begin{equation*}
R^{\mu_1}(\cdot):=\IX R^y(\cdot)\, \mu_1(dy)\in\PO
\end{equation*}
we would obtain a formulation analogous to~\eqref{eq-kanto-classique}.
\end{remark}

\begin{remark}
We didn't use any specific property of the Kolmogorov semigroup. The  dual equality can be generalized, without changing its proof, to any reference path measure $R \in \PO$ on  any Polish state space $\mathcal X$. 
\end{remark}

\section{Benamou-Brenier formulation of the entropic cost}
\label{sec-sbb}

We derive  some analogue of the Benamou-Brenier formulation (\ref{bena}) for the entropic cost.

\begin{theorem}[Benamou-Brenier formulation of the entropic cost]\label{teo:teo}
Let $V,\mu_0$ and $\mu_1$ be such that hypothesis (Exi), (Reg1) and (Reg2) stated in Section~\ref{sec-schrodinger} are satisfied. We have
\begin{equation}\label{eq:aaa}
\mathcal{A}(\mu_0, \mu_1) = H(\mu_0|m) + \inf _{ (\nu,v)} \IiZ \frac{|v_t(z)|^2}{2} \, \nu_t(dz) dt, 
\end{equation}
where the infimum is taken over all $(\nu_t, v_t) _{ 0\le t\le 1}$ such that, $\nu_t(dz)$ is identified with its
  density with respect to Lebesgue measure $\nu(t,z):= d\nu_t/dz$,  satisfying $\nu_0=\mu_0$, $\nu_1=\mu_1$ and the following continuity equation 
\begin{equation}\label{equa}
\partial_t\nu+\nabla \cdot \left(\nu\left[v- \nabla(V+\log \nu)/2 \right] \right)=0,
\end{equation}
 is satisfied  in a weak sense.
  \\
Moreover,  these results still hold true when the infimum in~\eqref{eq:aaa} is taken among all $(\nu,v)$ satisfying~\eqref{equa}  and such that  $v$ is a gradient vector field, that is  $$v_t(z)=\nabla \psi_t(z),\quad 0\le t\le 1, z\in\ZZ,$$ for some  function $ \psi\in \mathcal{C} ^{ \infty}([0,1)\times\ZZ).$  
\end{theorem}

\begin{remarks}\ \begin{enumerate}[(1)]
\item
The continuity equation~\eqref{equa} is the linear Fokker-Planck equation 
\begin{equation*}
\partial_t\nu+\nabla \cdot \left(\nu\left[v- \nabla V/2 \right] \right)- \Delta\nu/2=0.
\end{equation*}
Its solution $(\nu_t) _{ 0\le t\le 1}$, with $v$ considered as a known parameter, is the time marginal flow $ \nu_t=P_t$ of a weak solution $P\in\PO$ (if it exists) of the stochastic differential equation
\begin{equation*}
dX_t=[v_t(X_t)-\nabla V(X_t)/2]\,dt + dW^P_t,\quad P\textrm{-a.s.}
\end{equation*}
where $W^P$ is a $P$-Brownian motion, $P_0=\mu_0$   and $(X_t) _{ 1\le t\le 1}$ is the canonical process.
\item
 Clearly, one can restrict  the infimum in the identity~\eqref{eq:aaa} to $(\nu,v)$ such that    
\begin{equation}\label{eq-majoration-energie}
\IiZ  |v_t(z)|^2\, \nu_t(dz)dt < \infty.
\end{equation}

\end{enumerate}\end{remarks}

\begin{proof}
 Because of~\eqref{eq-14} and~\eqref{eq-13}, all we have to show is
 \begin{equation*}
\inf \{  H(P|R _{ \mu_0}); \  P\in\PO: P_0= \mu_0, P_1= \mu_1\}=\inf _{ (\nu,v)} \IiZ \frac{|v_t(z)|^2}{2} \, \nu_t(dz) dt,
 \end{equation*}
 where $(\nu, v)$ satisfies~\eqref{equa}, $\nu_0= \mu_0$ and $ \nu_1= \mu_1$.
As $R _{ \mu_0}$ is Markov, by \cite[Prop.~2.10]{leonard14} we can restrict the infimum to the set of all  Markov measures $P \in \PO$ such that $P_0=\mu_0, P_1=\mu_1$ and $H(P|R _{ \mu_0})<\infty$. For each such Markov measure $P$, Girsanov's theorem (see for instance \cite[Thm.\,2.1]{leo12} for a proof related to the present setting) states that there exists a measurable vector field $\beta^P_t(z)$ such that 
\begin{equation}\label{eq-15}
dX_t=[\beta^P_t(X_t)-\nabla V(X_t)/2]\,dt+ \,dW^P_t,\quad P\as,
\end{equation}
where $W^P$ is a $P$-Brownian motion.
Moreover, $\beta^P$ satisfies 
$
E_P \int_0^1 |\beta^P_t|^2(X_t)dt < \infty 
$
and
\begin{equation}\label{eq:identita}
H(P|R _{ \mu_0}) = \frac{1}{2} \IiZ |\beta^P_t|^2(z) \,P_t(dz) dt.
\end{equation}
For any $P$ with $P_0= \mu_0$, $H( \mu_0|m)< \infty$ and $H(P|R _{ \mu_0})< \infty$, we have $P_t\ll R_t=m\ll \Leb$ for all $t$. Taking $ \nu=(P_t) _{ 0\le t\le 1}$ and $ v=\beta^P$, the stochastic differential equation~\eqref{eq-15} gives~\eqref{equa}
 and optimizing the left hand side of~\eqref{eq:identita} leads us to
 \begin{equation*}
 \inf \{  H(P|R _{ \mu_0}); \  P\in\PO: P_0= \mu_0, P_1= \mu_1\}
 \le\inf _{ (\nu,v)} \IiZ \frac{|v_t(z)|^2}{2} \, \nu_t(dz) dt.
 \end{equation*}
On the other hand, it is proved in \cite{Zam86,leonard14} that the solution $P$ of the Schrödinger problem~\eqref{eq-08} is such that~\eqref{eq-15} is satisfied with $ \beta^P_t(z)=\nabla \psi_t(z)$ where $ \psi$ is given in~\eqref{eq-16}. This completes the proof of the theorem.
\end{proof}

\begin{corollary}\label{res-03}
Let $V$, $\mu_0$ and $\mu_1$ be such that the hypotheses stated   in Section~\ref{sec-schrodinger} are satisfied. We have
\begin{alignat}{1}\label{eq-17}
\mathcal{A}(\mu_0, \mu_1) =\ud [H(\mu_0|m) +&H( \mu_1|m)]\\
&+ \inf _{ ( \rho,v)}  \IiZ  \left(\ud |v_t(z)|^2 + \frac{1}{8}|\nabla \log \rho_t(z)|^2\right)  \rho_t(z)\, m(dz) dt, \nonumber
\end{alignat}
where the infimum is taken over all $(\rho_t,v_t) _{ 0\le t\le 1}$ such that $\rho_0\,m=\mu_0$, $ \rho_1\, m=\mu_1$ and the following continuity equation 
\begin{equation}\label{eq-18}
\partial_t \rho+e ^{ V}\nabla \cdot \left(e ^{ -V} \rho v \right)=0
\end{equation}
 is satisfied  in a weak sense.
  \\
Moreover,  these results still hold true when the infimum in~\eqref{eq-17} is taken among all $(\nu,v)$ satisfying~\eqref{eq-18}  and such that  $v$ is a gradient vector field, that is  $$v_t(z)=\nabla \theta_t(z),\quad 0\le t\le 1, z\in\ZZ,$$ for some  function $ \theta\in \mathcal{C} ^{ \infty}([0,1)\times\ZZ).$  
\end{corollary}

\begin{remark}
The density $ \rho$ in the statement of the corollary must be understood as a density $ \rho= d \nu/dm$ with respect to the reversing measure $m$. Indeed, with $ \nu(t,z)=d \nu_t/dz$, we see that $ \nu= e ^{ -V} \rho$ and the evolution equation~\eqref{eq-18} writes as the current equation $ \partial_t \nu+\nabla\cdot ( \nu v)=0.$
\end{remark}

This result was proved recently by Chen, Georgiou and Pavon in \cite{CGP14} in the case where $V=0$ without any mention to gradient type vector fields. The present proof is essentially the same as in \cite{CGP14}: we take advantage of the time reversal invariance of the relative entropy $H(\cdot|R)$ with respect to the reversible path measure $R.$

\begin{proof}
The proof follows almost the same line as Theorem~\ref{teo:teo}'s one. The additional ingredient is the  time-reversal invariance~\eqref{eq-19}: $H(P|R)=H(P^*|R)$. Let $P\in\PO$ be the solution of~\eqref{eq-07}. We have already noted that $P^*$ is the solution of the Schrödinger problem where the marginal constraints $ \mu_0$ and $ \mu_1$ are inverted. We obtain
\begin{alignat*}{2}
dX_t&=v^P_t(X_t)\,dt+dW^P_t,\qquad&& P\as\\
dX_t&=v ^{ P^*}_t(X_t)\,dt+dW ^{ P^*}_t,&& P^*\as
\end{alignat*}
where $W^P$ and $W ^{ P^*}$ are respectively Brownian motions with respect to $P$ and $P^*$ and 
\begin{eqnarray*}
H(P|R)&=&H( \mu_0|m)+ E_P \ud\int_0^1 | \beta^P_t(X_t)|^2\,dt\\
H(P^*|R)&=& H( \mu_1|m)+\ud E _{ P^*} \int_0^1 | \beta ^{ P^*}_t(X_t)|^2\,dt
	= H( \mu_1|m)+\ud E _{ P} \int_0^1 | \beta ^{ P^*} _{ 1-t}(X_t)|^2\,dt
\end{eqnarray*}
with $ v^P=-\nabla V/2+ \beta^P$ and $v ^{ P^*}=-\nabla V/2+ \beta ^{ P^*}.$ Taking the half sum of the above equations, the identity  $H(P|R)=H(P^*|R)$ implies that
\begin{equation*}
H(P|R)=\ud[H( \mu_0|m)+H( \mu_1|m)]
	+  \frac{1}{4}E_P\int_0^1 (| \beta^P_t|^2+| \beta _{ 1-t} ^{ P^*}|^2)\,dt.
\end{equation*}
Let us introduce the current velocities of $P$ and $P^*$ defined by
\begin{alignat*}{3}
v ^{ cu,P}_t(z)&:=v^P_t(z)-\ud\nabla\log  \nu^P_t(z)\ 
	&=&\  \beta^P_t(z)-\ud\nabla\log  \rho^P_t(z)\\
v ^{ cu,P^*}_t(z)&:=v ^{ P^*}_t(z)-\ud\nabla\log  \nu ^{ P^*}_t(z)\ 
&=&\ \beta ^{ P^*}_t(z)-\ud\nabla\log  \rho ^{ P^*}_t(z)
\end{alignat*}
where for any $0\le t\le 1, z\in\ZZ,$  
$$ 
\nu^P_t(z):= \frac{dP_t}{dz},\ \rho^P_t(z):= \frac{dP_t}{dm}(z)\quad \textrm{and}\quad \nu ^{ P^*}_t(z):= \frac{dP^*_t}{dz},\ \rho_t ^{ P^*}(z):= \frac{dP^*_t}{dm}(z).
$$ 
The naming {\it current velocity} is justified by the current equations
\begin{alignat*}{5}
& \partial_t  \nu^P+ \nabla\cdot(\nu^P v ^{ cu,P})&=&\ 0 &\quad \textrm{and}\quad &  \partial_t  \rho^P+e^V \nabla\cdot(e ^{ -V} \rho^P v ^{ cu,P}) &= \ 0,\\
 &\partial_t  \nu ^{ P^*}+ \nabla\cdot(\nu ^{ P^*} v ^{ cu,P^*})\ &=&\ 0 &\quad \textrm{and}\quad &
 \partial_t  \rho ^{ P^*}+e^V \nabla\cdot(e ^{ -V} \rho ^{ P^*} v ^{ cu,P^*})\ &= \ 0.
 \end{alignat*}
 To see that  the first equation $ \partial_t  \nu^P+ \nabla\cdot(\nu^P v ^{ cu,P})=0$ is valid, remark that $ \nu^P$ satisfies the Fokker-Planck equation~\eqref{equa} with $v$ replaced by $ \beta^P.$ The  equation for $ \rho^P$  follows immediately and the equations for $ \nu ^{ P^*}$ and $ \rho ^{ P^*}$ are derived similarly.
\\
 The very definition of $P^*$ implies that $ \rho ^{ P^*}_t= \rho^P _{ 1-t}$ and the time reversal invariance $R^*=R$ implies that   
$$
v ^{ cu,P^*}_t(z)=-v ^{ cu,P}_{ 1-t}(z),\quad 0\le t\le 1, z\in\ZZ.
$$
Therefore, $ \beta ^{ P^*}_{ 1-t}=-v ^{ cu,P}_t+\ud\nabla\log \rho^P_t$ and 
$ \frac{1}{4}(| \beta_t ^P|^2+| \beta _{ 1-t} ^{ P^*}|^2)= \ud| v_t ^{ cu,P}|^2+ \frac{1}{8}|\nabla\log \rho_t^P|^2.$ This completes the proof of the first  statement of the corollary. 

For the second statement about $v=\nabla \theta,$ remark that as in Theorem~\ref{teo:teo}'s proof,  the solution $P$ of the Schrödinger problem is such that $ \beta^P=\nabla \psi$ for some smooth function $ \psi.$  One concludes with $v ^{ cu,P}= \beta^P-\ud\nabla \log \rho^P,$ by taking $ \theta= \psi-\log \sqrt{\rho^P}.$
\end{proof}

\begin{remarks}\ \begin{enumerate}[(1)]
\item
The current velocity $ v ^{ cu,P}$ of a diffusion path measure $P$ has been introduced by Nelson in \cite{nelson} together with its {\it osmotic velocity} $v ^{ os,P} := \ud \nabla\log \rho^P.$  
\item
Up to a multiplicative factor,  $\IZ |\nabla \log \rho_t(z)|^2\,  \rho_t(z)\, m(dz)$ is the entropy production or Fischer information. The average osmotic action is 
$A ^{ os}( P):=\IiZ \ud| v ^{ os,P}|^2\, dP_tdt=\IiZ  \frac{1}{8}|\nabla \log \rho|^2  \rho\, dm\,dt$ and it is directly connected to a variation of entropy. It's worth  remarking that by considering the dilatation in time of the reference path measure as introduced in Remark~\ref{rem-entropic-wasserstein}, the osmotic action vanishes in the limit for $\varepsilon \to 0$.
Let us define now the osmotic cost  
$$
I ^{ \mathrm{os}}( \mu_0, \mu_1):= \inf \{  A ^{ os}(P); P\in\PO:P_0= \mu_0, P_1=\mu_1\}
$$ 
and the current cost 
$I ^{ \mathrm{cu}}( \mu_0, \mu_1):= \inf _{ ( \rho,v)}  \IiZ \ud |v_t(z)|^2 \,  \rho_t(z)\, m(dz) dt$ where the infimum runs through all the $( \rho,v)$  satisfying~\eqref{eq-18}.  The standard Benamou-Brenier formula precisely states that $I ^{ \mathrm{cu}}( \mu_0, \mu_1)=W^2_2( \mu_0, \mu_1)/2.$  Therefore, Corollary~\ref{res-03} implies that
$$
\mathcal{A}(\mu_0, \mu_1) \ge\ud [H(\mu_0|m) +H( \mu_1|m)]+\ud W_2^2( \mu_0, \mu_1)+ I ^{ os}( \mu_0, \mu_1).
$$
\end{enumerate}\end{remarks}
 In particular, by the positivity of the entropic cost $I^{os}$ we obtain the following relation between the entropic  and  Wasserstein costs:
\begin{corollary}\label{cor-super}
Let $V,\mu_0$ and $\mu_1$ be such that the hypotheses stated in Section~\ref{sec-schrodinger} are satisfied. We have
$$
\mathcal{A}(\mu_0, \mu_1) \ge\ud [H(\mu_0|m) +H( \mu_1|m)]+\ud W_2^2( \mu_0, \mu_1).
$$
\end{corollary}

\section{Contraction with respect to the entropic cost}
\label{sec-contraction}

The analogy between optimal transport and minimal entropy can also be observed in the context of contractions.

As explained in the introduction, contraction in Wasserstein distance depends on the curvature. Even if there are actually many contraction inequalities in Wasserstein distance, we focus here on  two main results. The first one depends on the curvature and the second one includes the dimension. These results can be written for more general semigroups satisfying the curvature-dimension condition as  defined in the  Bakry-\'Emery theory. 

\bigskip

In the context of the  Kolmogorov semigroup of Section~\ref{sec-kolmo} with a generator given by~\eqref{eq-11} in $\ZZ$,  the two main contraction inequalities can be formulated as follows. 

\medskip
\begin{itemize}
\item   Let assume that for some real $\lambda,$ we have  ${\rm Hess}(V)\geq {\lambda}\,{\rm Id}$  in the sense of symmetric matrices. Then for any $f,g$ probability densities with respect to the measure $m$ and any $t\geq0$, 
\begin{equation}
\label{eq-contr-w}
W_2(T_t f\,m, T_t g\,m) \leq e^{-\frac{\lambda}{2} t}W_2(fm, gm).
\end{equation}
Let us recall that this result was proved in~\cite{renesse-sturm} in the general context of Riemannian manifold. Although in the context of Kolmogorov semigroups the proof is easy, its generalization for the entropic cost to a Riemannian setting is not trivial. 

\item   When $L=\Delta/2$ that is $V=0$, the heat equation in $\ZZ$ satifies the following  dimension dependent  contraction property: 
\begin{equation}\label{dim-contr}
W_2^2(T_t f \,\Leb, T_s g\, \Leb) \leq W_2^2(f\,\Leb, g\,\Leb) + n(\sqrt t - \sqrt s)^2,
\end{equation}
for any  $s,t\geq0$ and any $f,g$ probability densities with respect to  the Lebesgue measure $\Leb$.  This contraction was proved in a more  general context in \cite{bgl2,kuwada}.
\end{itemize}

The two inequalities~\eqref{eq-contr-w} and~\eqref{dim-contr} can be proved in terms of entropic cost. Let us choose the reference path measure $R$  associated with the potential $V$ and take $\varepsilon , u>0$ and   $\mu_0,\mu_1\in\PZ$.  In order to extend  for each $u,\varepsilon>0$ the dual formulation for the entropic cost of Theorem~\ref{res-02}, consider the semigroup $(T_{\varepsilon ut})_{t\geq0}$  and the corresponding path measure $R ^{ \epsilon u}$:  time is dilated by the factor $( \epsilon u) ^{ -1}$. Theorem~\ref{res-02} implies that  
\begin{equation*}
 \mathcal{A} ^{ R ^{ \epsilon u}}(\mu_0, \mu_1)= H( \mu_0|m)
+\sup\left\{\IZ \psi\, d \mu_1-\IZ \log T_{\varepsilon u}(e^{\psi})\, d \mu_0,\quad  \psi\in C_{b}(\ZZ)\right\}. 
\end{equation*}
Now by changing $\psi$ with $\psi/\varepsilon$ we see that  
\begin{equation}\label{duep}
\varepsilon\mathcal A^{ R ^{ \epsilon u}}(\mu_0, \mu_1)=\varepsilon H( \mu_0|m)
+\sup\left\{\IZ \psi\, d \mu_1-\IZ \mathcal{Q}^{ \varepsilon }_{u} \psi\, d \mu_0,\quad  \psi\in C_{b}(\ZZ)\right\}
\end{equation}
where for any  $\psi\in C_{b}(\ZZ)$, 
\begin{equation}
\label{def-hje}
\mathcal{Q}^{ \varepsilon }_{u} \psi=\varepsilon \log T_{\varepsilon u}(e^{\psi/\varepsilon}).
\end{equation}
For simplicity, we denote $\varepsilon\mathcal A^{R ^{ \varepsilon u}}=\mathcal A^{\varepsilon }_u$ and $\mathcal A^{\varepsilon}_1 = \mathcal A^{\varepsilon}$. 
\\
As  explained in Remark~\ref{rem-entropic-wasserstein}, we have 
\begin{equation}\label{gamma}
\lim_{\varepsilon \to 0} \mathcal A^{\varepsilon}_u(\mu_0, \mu_1) = W_2^2(\mu_0, \mu_1)/2u.
\end{equation}

The  entropic cost associated to the Kolmogorov semigroup has the following properties.

\begin{theorem}[Contraction in entropic cost]\label{thm-contr}
Let $\varepsilon>0$ be fixed.  
\begin{enumerate}[(a)]
\item If $V$ satisfies ${\rm Hess}(V)\geq \lambda\,{\rm Id}$ for some $\lambda\in\R$,  then for any $t \geq 0$,
\begin{equation}\label{contr-entr}
\mathcal A^{\varepsilon}_b (T_{u_t(b)} fm, T_t gm)\leq \mathcal A^{\varepsilon}_{v_t(b)} (fm,gm) + \varepsilon[H(T_{u_t(b)} fm|m) - H(fm|m)],
\end{equation}
where $f,g$ are probability densities with respect to $m$, and 
\begin{equation}
\label{def-uv}
\begin{array}{ll}
u_t(b)= t+\frac{1}{\lambda} \log \left(\frac{e^{-\varepsilon\lambda b}}{ 1+ e^{\lambda t}(e^{-\varepsilon \lambda b} - 1)}\right) & \quad  v_t(b) = -\frac{1}{\lambda \varepsilon} \log (1+ e^{\lambda t}(e^{-\varepsilon \lambda b} - 1)) 
\end{array}
\end{equation}
where:  if $\lambda \leq 0$, $b \in (0, \infty)$ and 
 if $\lambda > 0$, $b \in (0, -\frac{1}{\lambda \varepsilon}\log (1- e^{-\lambda t})).$
\item If $V=0$ then for any $t \geq 0$,
$$
\mathcal{A}^{\varepsilon} (T_t fm, T_s gm) \leq \mathcal A^{\varepsilon}(fm, gm) + \frac{n}{2}(\sqrt t - \sqrt s)^2 + \varepsilon [H(T_t fm|m)- H(fm|m)].
$$
\end{enumerate}
\end{theorem}
The proof of this theorem relies on the following commutation property between the Markov semigroup $T_t$ and the semigroup $\mathcal Q^{\varepsilon}_t$ defined at~\eqref{def-hje}. Let us notice that the second statement of  next lemma was proved in \cite[Section~5]{bgl2}. 

\begin{lemma}[Commutation property]\label{el-comm}
\label{lem-commutation}
Let $s, t \geq 0$, $\varepsilon > 0$ and $f: \ZZ \to \mathbb{R}$ be any bounded measurable  function.
\begin{enumerate}[(a)] 
\item If ${\rm Hess}(V)\geq \lambda\,{\rm Id}$ for some real $ \lambda,$ then
\begin{equation}\label{qu}
\mathcal{Q}^{\varepsilon}_{v_t(b)}(T_tf)\leq T_{u_t(b)}(\mathcal{Q}_b^{\varepsilon} f)
\end{equation}
where for each $t \geq 0,$ the numbers  $u_t(b)$, $v_t(b)$ and $b$ are given in~\eqref{def-uv}.\\
 Moreover, for $\varepsilon$ small enough and $t>0$ fixed, \eqref{qu} is valid for all  $b$ positive. 
\item If $V=0$ then 
$$
\mathcal{Q}^{\varepsilon}_1 (T_t f)\leq T_s(\mathcal{Q}^{\varepsilon}_1 f)+\frac{n}{2}(\sqrt t - \sqrt s)^2.
$$
\end{enumerate}
\end{lemma}

\begin{proof}
We only have to  prove the first statement  (a).
Let us define for each $s\leq t$ the function
$$
\Lambda (s) = T_\alpha \mathcal Q^{\varepsilon}_\beta (T_{t-s}f)
$$
with $\alpha: [0,t] \to [0, \infty)$ an increasing function such that $\alpha(0)=0$, and $\beta: [0,t] \to [0, \infty)$   and we call $\beta(t)=b$. Setting $g = \exp(T_{t-s}f/\varepsilon)$, using  the chain rule for the diffusion operator $L$ we obtain
\begin{equation}
\begin{split}
\Lambda'(s)& = \varepsilon T_\alpha \left[\alpha' L \log T_{\varepsilon \beta}g + \frac{1}{T_{\varepsilon \beta}g}T_{\varepsilon \beta}\left(\varepsilon \beta' Lg -gL\log g\right)\right]\\
& =  \varepsilon T_\alpha \left[\alpha' \left(\frac{LT_{\varepsilon \beta}g}{T_{\varepsilon \beta}g} - \frac{|\nabla T_{\varepsilon \beta}g|^2}{2(T_{\varepsilon \beta}g)^2}\right)  + \frac{1}{T_{\varepsilon \beta}g}T_{\varepsilon \beta}\left(\varepsilon \beta' Lg -Lg+ \frac{|\nabla g|^2}{2g}\right)\right]\\
& = \varepsilon T_\alpha \left[\frac{1}{T_{\varepsilon \beta}g}\left(LT_{\varepsilon \beta}g(\alpha'+\varepsilon \beta'-1) +T_{\varepsilon \beta}\left(\frac{|\nabla g|^2}{2g} \right)-\alpha' \frac{|\nabla T_{\varepsilon \beta}g|^2}{2T_{\varepsilon \beta}g}\right)\right]\\
& \geq \varepsilon T_\alpha \left[\frac{1}{T_{\varepsilon \beta}g}\left(LT_{\varepsilon \beta}g(\alpha'+\varepsilon \beta'-1) +\frac{1}{2}T_{\varepsilon \beta}\left(\frac{|\nabla g|^2}{g} \right)(1-e^{-\lambda \varepsilon \beta}\alpha')\right)\right]\\
\end{split}
\end{equation}
where the last inequality is given by the commutation,
$$
\frac{|\nabla T_t g|^2}{T_t g} \leq e^{-\lambda t}T_t\left(\frac{|\nabla g|^2}{g}\right)
$$
which is implies  by the condition ${\rm Hess}(V)\geq \lambda{\rm Id}$ (see for instance~\cite[Section~3.2]{bgl}). If the following conditions on $ \alpha$ and $ \beta$  hold
 \begin{equation}\label{syst}
 \left\{
 \begin{array}{l}
 \alpha'+\varepsilon \beta'-1=0\\
 1-e^{-\lambda \varepsilon \beta}\alpha'=0,
 \end{array}\right.
 \end{equation}
we have $\Lambda'(s)\geq 0$ for each $0\leq s \leq t$. In particular $\Lambda(0)\leq \Lambda(t)$ for each $t\geq 0$, that is
$$
\mathcal{Q}^{\varepsilon}_{v_t(b)}(T_tf)\leq T_{u_t(b)}(\mathcal{Q}^{\varepsilon}_b f)
$$
where $v_t(b)=\beta(0)$ and $u_t(b)=\alpha(t)$. Finally solving system (\ref{syst}) together with the conditions $\alpha(0)=0$, $\beta(t)=b$, we can compute the explicit formulas for $v$ and $u$ as in  statement (a). In particular, substituting $\alpha'$ in the second equation of the system and integrating from 0 to $t$ we obtain the following relation
\begin{equation}\label{star}
e^{-\varepsilon \lambda \beta(0)}= 1+ e^{\lambda t}(e^{-\varepsilon \lambda b}-1).
\end{equation}
If we assume for a while that the term on the right hand side is positive, we obtain 
$$
\beta(0)= -\frac{1}{\lambda \varepsilon} \log (1+ e^{\lambda t}(e^{-\varepsilon \lambda b} - 1))
$$
and 
$$
\alpha(t)= t+\frac{1}{\lambda} \log \left(\frac{e^{-\varepsilon\lambda b}}{ 1+ e^{\lambda t}(e^{-\varepsilon \lambda b} - 1)}\right).
$$
Let us study now the sign of the right hand side in~\eqref{star}. 
\begin{itemize}
\item If $\lambda \leq 0$, it is positive for each $b \in \mathbb{R}$;
\item If $\lambda > 0$ in order to be positive, we need the condition for $b$,
\begin{equation*}
b<-\frac{1}{\varepsilon \lambda}\log(1-e^{-\lambda t}) := b_0.
\end{equation*}
Finally let us consider the case when $\varepsilon>0$ is small. From~\eqref{star} we obtain the relation, 
\begin{equation*}
\beta (0)= b e^{\lambda t} + o(\varepsilon)
\end{equation*}
for each $\lambda \in \mathbb{R}$ and $b$ positive.
\end{itemize} 
This completes the proof of the lemma.
\end{proof}

\begin{proof}[Proof of Theorem~\ref{thm-contr}]
The proof is based on the dual formulation stated in Theorem~\ref{res-02}. Let $\psi\in C_{b}(\ZZ)$, by Lemma~\ref{lem-commutation} under the condition ${\rm Hess}(V)\geq \lambda{\rm Id}$ and by time reversibility, 
\begin{displaymath}
\begin{split}
\IZ \psi\, T_tg\, dm - \IZ \mathcal Q^{\varepsilon}_b \psi\, T_{u_t(b)} f \, dm & = \IZ T_t \psi\, g\, dm - \IZ T_{u_t(b)}\mathcal{Q}^{\varepsilon}_b \psi\, f \, dm \\
& \leq \IZ T_t \psi\, g\, dm - \IZ \mathcal{Q}^{\varepsilon}_{v_t(b)} T_t \psi\, f \, dm\\
& \leq \mathcal{A}^{\varepsilon}_{v_t(b)}(fm,gm) - \varepsilon H(fm|m).
\end{split}
\end{displaymath}
Finally taking the supremum over $\psi\in C_{b}(\ZZ)$ we obtain the desired inequality in $(i)$. 
The same argument can be used to prove the contraction property in $(ii)$, applying  the second commutation inequality in Lemma~\ref{lem-commutation}.
\end{proof}

\begin{remark}
Let observe that if $\lambda < 0$, the function $\beta(s)$, for $s \in [0,t]$, is decreasing, while for $\lambda > 0$ it is increasing and if $\lambda=0$ it is the constant function $\beta(t)=b$. In particular by choosing $b=1$~\eqref{qu} writes as follows
 \begin{equation*}
\mathcal Q^{\varepsilon}_1(T_t f) \leq T_t (\mathcal Q^{\varepsilon}_1f).
\end{equation*}
\end{remark}

\begin{remark}
Lemma~\ref{lem-commutation} can be proved in the general context of a Markov diffusion operator under the Bakry-\'Emery curvature-dimension condition. Its application to more general problems is actually a working paper of the third author. 
\end{remark}

\begin{remarks}
Let us point out two converse assertions.
\begin{enumerate}[(1)]
\item The contraction in entropic cost in Theorem~\ref{thm-contr} implies back the contraction in Wasserstein cost. Indeed, under the assumptions of Section~\ref{sec-schrodinger}, it can be easily checked that when $\varepsilon \to 0$, we have $u(t) \to t$ and $v(t) \to be^{\lambda t}$. Therefore,  with   (\ref{gamma})  and (\ref{contr-entr}), one  recovers (\ref{eq-contr-w}). Analogous arguments can be applied to recover the contraction of the Wasserstein cost~(\ref{dim-contr}) when $V=0$. 
\item The commutation property in Lemma~\ref{lem-commutation} implies back the convexity of the potential $V$. This can be seen by  differentiating~\eqref{qu} with respect to  $b$  around 0. We believe also that for $\varepsilon>0$ fixed, inequality~\eqref{contr-entr} implies back the convexity of the potential. 
\end{enumerate}
\end{remarks}

\section{Examples}
\label{sec-examples}
In this section we will compute explicitly the results discussed in the previous sections, between two given measures. We first compute the Wasserstein cost, its dual and Benamou-Brenier formulations and the displacement interpolation, as exposed in the introduction. Then, we'll do the same for the entropic cost, taking in consideration two different reference path measures $R$. In particular, we'll compute~(\ref{duep}), for $u=1$ and $\varepsilon >0$ and look at the behavior in the limit $\varepsilon \to 0$ to recover the classical results of the optimal transport. 
For abuse of notation we will denote with $\mu_t$ both the interpolation and its density with respect to the Lebesgue measure $dx$.
We introduce for Gaussian measures the following notation:  for any  $m \in \ZZ$ and $v \in \R$, the density with respect to the Lebesgue measure of $\N(m,v^2)$ is given by 
$$
(2\pi v^2)^{-n/2}\exp\Big(-\frac{|x-m|^2}{2v^2}\Big),
$$
As marginal measures we consider for $x_0, x_1 \in \mathbb{R}^n$
\begin{equation}\label{eq:cond}
\begin{array}{ll}
\mu_{0}(x) := \N(x_0,1), &  \mu_{1}(x) := \N(x_1,1).
\end{array}
\end{equation}

Note that the entropic interpolation between two Dirac measures $\delta_x$ and $\delta_y$ should be the Bridge $R^{xy}$ between $x$ and $y$ with respect to the reference measure $R$. But unfortunately $H(\delta_x|m),H(\delta_y|m)=\infty$, hence we consider only marginal measures with a density with respect to $m$. 

\subsection{Wasserstein cost}
\label{subsec-was}
The Wasserstein cost between $\mu_0, \mu_1$ as in (\ref{eq:cond}), is
\begin{equation*}
W_2^2(\mu_0, \mu_1)=d(x_0, x_1)^2.
\end{equation*}
In its dual formulation, the supremum is reached by the function 
$$
\psi(x)= (x_1-x_0)x
$$
and in the Benamou-Brenier formulation the minimizer vector field is 
$$
v^{MC}= x_1-x_0
$$
The displacement or McCann interpolation is given by 
\begin{equation}\label{eq:mccan}
\mu_{t}^{MC}(x) = \N(x_t,1)
\end{equation}
where $x_t=(1-t)x_0+tx_1.$
In other words using the push-forward notation~\eqref{mccan},
$$
\mu_{t}^{MC}(x) = (\hat x_t^{MC})_\#\mu_0
$$
with $\hat x_t^{MC}(x):= (1-t)x +t(x+ x_1- x_0)$ a trajectory whose associated velocity field is $v^{MC}=x_1-x_0.$

\subsection{Schrödinger cost}
\subsubsection*{Heat semigroup}

As a first example we consider on the state space $\mathbb{R}^{n}$  the Heat (or Brownian) semigroup, that corresponds to the case $V=0$ in our main example in Section~\ref{sec-kolmo}, whose infinitesimal generator is the laplacian $L=\Delta/2$ and the invariant reference measure is the Lebesgue measure $dx$.  Since we're interested in the $\varepsilon-$entropic interpolation, with $\varepsilon>0$, we take in consideration the Heat semigroup with a dilatation in time, whose density kernel is given by
$$
p_{t}^\varepsilon (x,y) =(2\pi \varepsilon t)^{-n/2} \exp \left(-\frac{|x-y|^2}{2\varepsilon t}\right) 
$$ 
i.e. $p_{t}^\varepsilon (x,y)=\N(y,\varepsilon t)$ for $t>0, (x,y) \in \mathbb{R}^{n}\times \mathbb{R}^{n}.$

\begin{itemize}
\item The entropic interpolation~\eqref{entropic} is
\begin{equation}\label{eq:heat}
\mu_{t}^\varepsilon(x) = \N(x_t,D_t^\varepsilon)
\end{equation}
where $x_t$ is like in (\ref{eq:mccan}) and $D_{t}^\varepsilon: [0,1] \to \mathbb{R}^+$ is a polynomial function given by
$$
D_{t}^\varepsilon = \alpha^{\varepsilon}t(1-t)+1
$$ 
with $\alpha^{\varepsilon}= \delta^2/(1+\delta)$ where $\delta=(\varepsilon-2+\sqrt{4+\varepsilon^2})/2$. We observe that $D_t^\varepsilon$ is such that $D_0=D_1=1$ with a maximum in $t=1/2$ for each $\varepsilon >0$, (see Figure~\ref{fig:variance}). 

 It's worth to point out how we managed to derive an explicit formula for the entropic interpolation. The key point is the resolution of the Schr\"odinger system~\eqref{eq-10} that in our example writes as 
$$
\left\{\begin{array}{ll}
(2\pi)^{-n/2} \exp \left(-\frac{|x-x_0|^2}{2}\right) &= f(x)\int g(y)(2\pi \varepsilon t)^{-n/2} \exp \left(-\frac{|x-y|^2}{2\varepsilon t}\right)\, dy\\
(2\pi)^{-n/2} \exp \left(-\frac{|x-x_1|^2}{2}\right) &=g(x)\int f(y)(2\pi \varepsilon t)^{-n/2} \exp \left(-\frac{|x-y|^2}{2\varepsilon t}\right) \,dy.
\end{array}\right.
$$
By taking $f$ and $g$ exponential functions of the type 
$$
e^{a_2x^2+a_1x+a_0} \quad \textrm{with} \quad a_0,a_1,a_2 \in \R
$$
we can solve the system explicitely by determining the coefficients of $f$ and $g$.

One can also express the entropic interpolation through the {\it push-forward} notation as introduced at~\eqref{eq-20}, $\mu_{t}^\varepsilon = (\hat x_{t}^\varepsilon)_\#\mu_{0}$ where 
$$
\hat x_{t}^\varepsilon(x)= \sqrt{D_{t}^\varepsilon}(x-x_0)+x_t.
$$
Furthermore $\hat x_{t}^\varepsilon$ satisfies the differential equation 
\begin{equation}\label{eq:edo}
\dot x_t^\varepsilon = v^{cu, \varepsilon}(x_t^\varepsilon)
\end{equation}
where $v^{cu, \varepsilon}$ is the current velocity, and is given by 

\begin{displaymath}
v^{cu, \varepsilon} = \frac{\dot D_{t}^\varepsilon}{2 D_t^\varepsilon}(x-x_t) + x_{1} - x_{0}.
\end{displaymath} 
It can be finally verified that the entropic interpolation (\ref{eq:heat}) satisfies the {\it PDE} 
\begin{equation}\label{eq:pde}
\dot \mu_t^\varepsilon + \nabla \cdot (\mu_t^\varepsilon v^{cu, \varepsilon}) = 0. 
\end{equation}
 
\begin{remark}\label{remm}
Let observe that if $x_0=x_1$, $\mu_t^\varepsilon$ is not constant in time, unlike the McCann interpolation.\
\end{remark}

\item Denoting $P\in\PO$ the path measure whose flow is given by (\ref{eq:heat}) and that minimizes $H(\cdot|R)$, the entropic cost between $\mu_0, \mu_1$ as in (\ref{eq:mccan}) is
\begin{equation*}
\mathcal A^{\varepsilon}_{u}(\mu_0, \mu_1)= H( P|R).
\end{equation*}
The easiest way to compute this quantity is to use the Benamou-Brenier formulation in Section~\ref{sec-sbb}. The resulting formula has not a nice and interesting form, therefore we don't report it explicitly.
\item In the dual formulation proved in Section~\ref{sec-schrodinger}, the supremum is reached by the function $\psi\in C_{b}(\ZZ)$ given, up to a constant term, by 
\begin{equation}\label{eq-psi}
\psi_t(x)= -\frac{1}{2} \frac{\delta}{1+\delta(1-t)}x^2-\frac{1}{2}\frac{\gamma}{1+\delta(1-t)}x
\end{equation}
where $\delta$ as in \eqref{eq:heat} and $\gamma=2[x_0(1+\delta)-x_1]$. 
\item In the Benamou-Brenier formulation in Theorem~\ref{teo:teo} the minimizer vector field is 
$$
v^{H}= \nabla \psi_t.
$$
where $\psi_t$ is given by (\ref{eq-psi}) and $\nabla \psi_t$ represents the forward velocity. It can be easily verified that the equation
\begin{equation}
\partial_t\mu_t+\nabla \cdot \left(\mu_t\left[ \nabla\psi_t-\frac{\nabla \mu_t}{2 \mu_t} \right] \right)=0
\end{equation}
is satisfied.

\end{itemize}

\subsubsection*{Ornstein-Uhlenbeck semigroup}

As a second example, we consider on the state space $\mathbb{R}^n$ the Ornstein-Uhlenbeck semigroup, that corresponds to the case $V=|x|^2/2$ for the Kolmogorov semigroup in Section~\ref{sec-kolmo}, whose infinitesimal generator is given by $L=(\Delta - x \cdot \nabla)/2$ and the invariant measure is the standard Gaussian in $\mathbb{R}^n$ . Here again we consider the kernel representation with a dilatation in time; in other words, for $\varepsilon >0$, the kernel with respect to the Lebesgue measure is given by

$$
p_{t}(x,y) = (2\pi (1-e^{-\varepsilon t}))^{-n/2} \exp \left(-\frac{|y-xe^{-\varepsilon t/2}|^2}{2(1-e^{-\varepsilon t})}\right) 
$$
i.e. $p_{t}(x,y)=\N(xe^{-\varepsilon t/2}, 1-e^{-\varepsilon t}).$

\begin{itemize}
\item The entropic interpolation~\eqref{entropic} is given by
\begin{equation}\label{eq:entrou}
\mu_{t}^{\varepsilon}(x) = \N(m_t, D_t^\varepsilon)
\end{equation}
where $m_t = a_t [(e^{-\varepsilon t/2}-e^{-\varepsilon(1-t/2)})x_0+ (e^{-\varepsilon(1-t)/2}- e^{-\varepsilon(1+t)/2})x_1]$ with 
$$
a_t := \frac{1+\delta-\delta e^{-\varepsilon}}{(1-e^{-\varepsilon})[\delta(1+\delta)(e^{-\varepsilon t}+e^{-\varepsilon(1-t)})-2\delta^{2}e^{-\varepsilon}]}
$$
with $\delta$ as in (\ref{eq-psi1}), and $D_t^{\varepsilon}:[0,1] \to \mathbb{R}^+$ defined as
$$
D_{t}^{\varepsilon} :=-1+2(1-e^{-\varepsilon})a_t
$$
satisfying, as in the case of the Heat semigroup, $D_{0}^{\varepsilon} =D_{1}^{\varepsilon}=1$. 

Furthermore, we have $\mu_t^{\varepsilon}  = (\hat x_t^{\varepsilon} )_\#\mu_0$
where
$$
\hat x^{\varepsilon} _t := \sqrt{D_t^{\varepsilon}}(x-x_0) + m_t.
$$

It can be verified that equations (\ref{eq:pde}) and (\ref{eq:edo}) hold true also in the Ornstein Uhlenbeck case, with the current velocity given by
\begin{equation}\label{eq:curr}
v^{\varepsilon}_{cu}  = \frac{\dot D_t^{\varepsilon}}{2D_t}(x-x_0)+\dot m_t
\end{equation}
\item The entropic cost between $\mu_0, \mu_1$ can be computed as in the Heat semigroup case by 
\begin{equation*}
\mathcal A^{\varepsilon}_{u}(\mu_0, \mu_1)= H( P|R)
\end{equation*}
where $ P$ is the path measure associated to the flow (\ref{eq:entrou}) which minimizes $H(\cdot|R)$.
\item In the dual formulation at Section~\ref{sec-schrodinger}, the supremum is reached, up to a constant term, by the function 
\begin{equation}\label{eq-psi1}
\psi_t(x)= -\frac{1}{2}\frac{\varepsilon\delta e^{-\varepsilon(1-t)}}{1+\delta(1-e^{-\varepsilon(1-t)})}x^2+\frac{\varepsilon\gamma e^{-\varepsilon(1-t)/2}}{1+\delta(1-e^{-\varepsilon(1-t)})}x
\end{equation}
where  $\delta=(e^{-\varepsilon}-\sqrt{e^{-2\varepsilon}-e^{-\varepsilon}+1})/(e^{-\varepsilon}-1)$ and $\gamma=(x_0 e^{-\varepsilon/2}-x_1(1+\delta-\delta e^{-\varepsilon}))/(1-e^{-\varepsilon})$. 
\item In the Benamou-Brenier formulation (Theorem~\ref{teo:teo}) the minimizer vector field is 
$$
v^{OU}= \nabla \psi_t.
$$

\end{itemize}

\begin{remark}
Let observe that both in the Heat and Ornstein-Uhlenbeck cases, if we take the limit $\varepsilon \to 0$ of the entropic interpolation, of the velocities $v^{H}, v^{OU}$ and of the function $\psi_t$, we recover the respective results for the Wasserstein cost stated in Subsection~\ref{subsec-was}.
\end{remark}

In the following figures we refer to the McCann interpolation with a dotted line, the Heat semigroup with a dashed line and the Ornstein Uhlenbeck semigroup with a continuous line. We fix $\varepsilon=1$ and consider marginal measures in one dimension. Figure~\ref{var} represents the variance of the three interpolations, independent from the initial and final means $x_0, x_1$. Figures~\ref{fig-1} and~\ref{fig-2} correspond to the mean in the three cases respectively with the initial and final means symmetric w.r.t the origin, and for any means. It's worth to remark from these images that the McCann interpolation and the entropic interpolation in the case of the heat semigroup, have the same mean. Finally figures~\ref{fig-3} and~\ref{fig-4} represent the three interpolations at time $t=0, 1/2, 1$ respectively with different marginal data, as before.

\begin{figure}[htbp]
\centering
\includegraphics[scale=0.4]{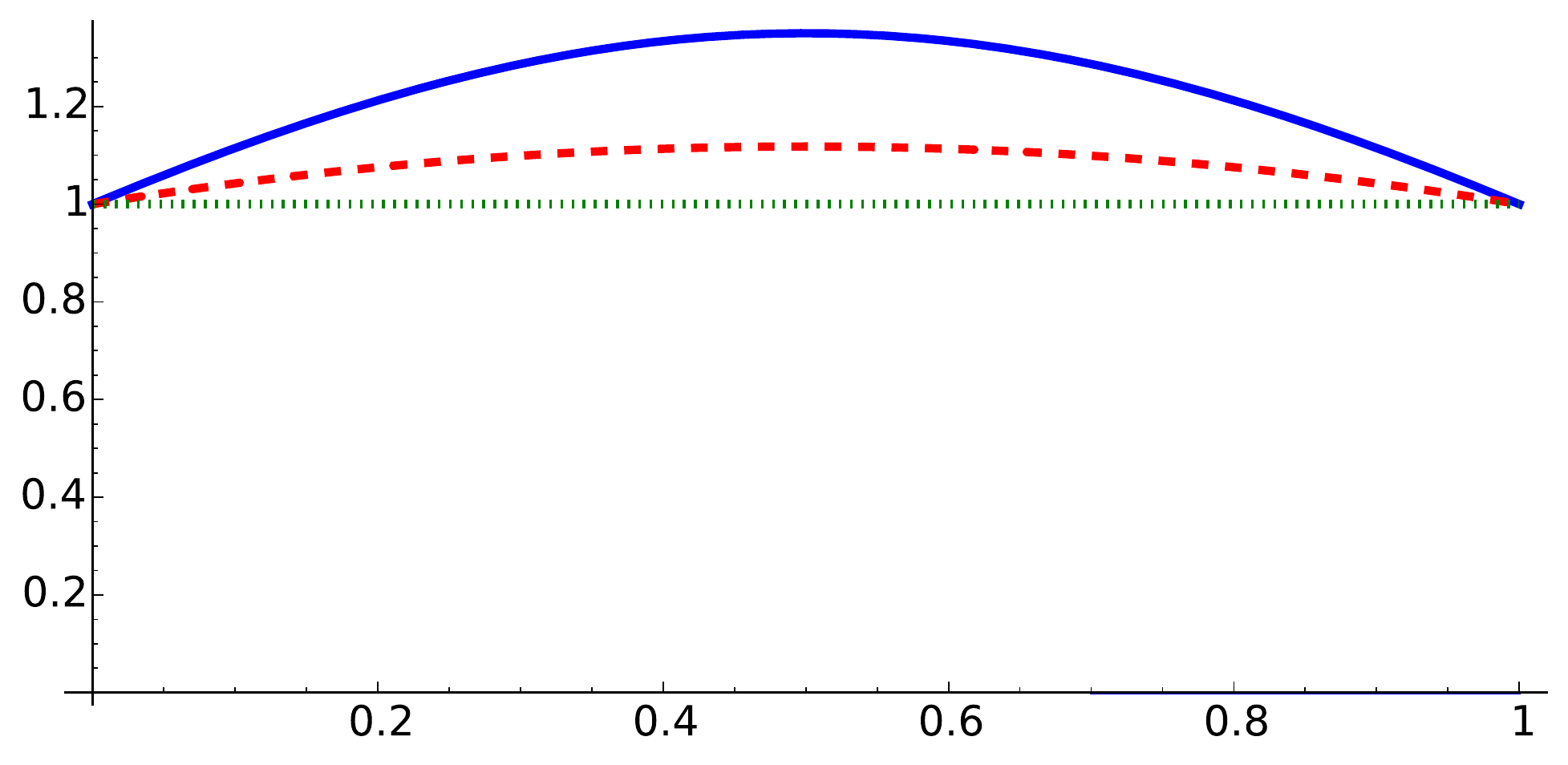}
\caption{Variance, $\varepsilon = 1$}\label{fig:variance}
\label{var}
\end{figure}

\begin{figure}[!h]
\begin{center}
\includegraphics[width=9cm]{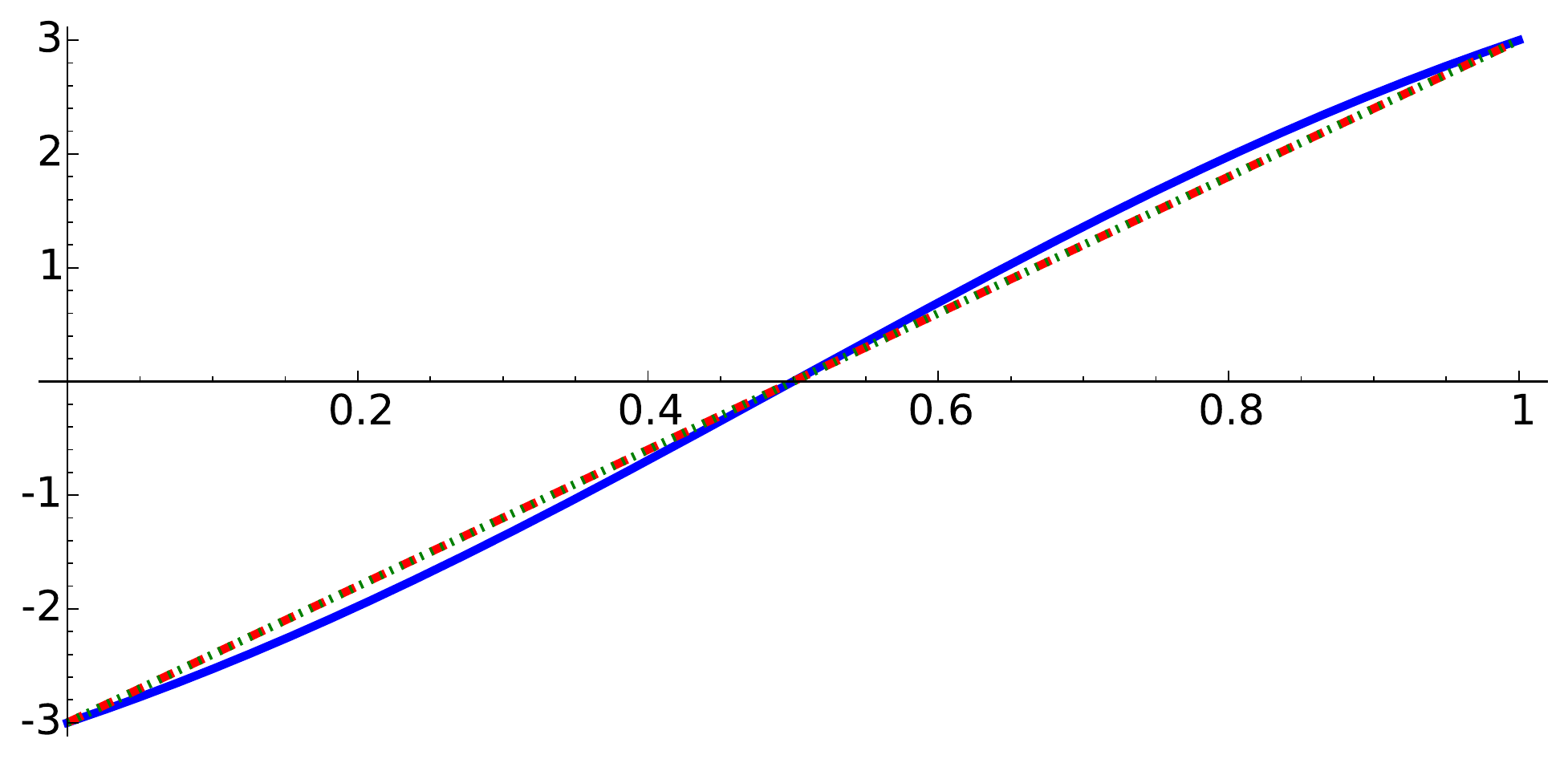}
\caption{Mean with $x_0=-3, x_1=3, \varepsilon = 1$}
\label{fig-1}
\end{center}
\end{figure}

\begin{figure}[!h]
\begin{center}
\includegraphics[width=9cm]{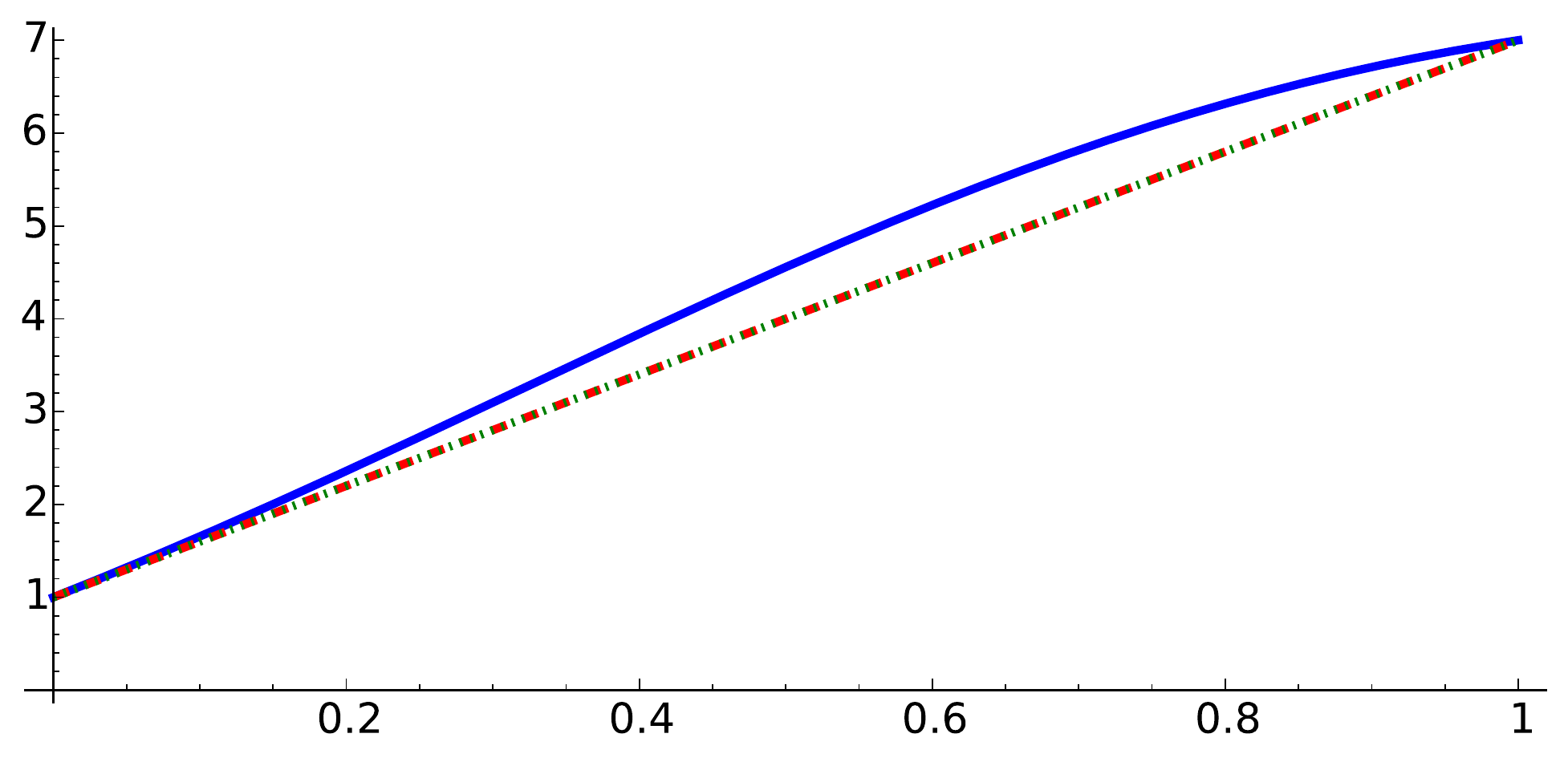}
\caption{Mean $x_0=1, x_1=7, \varepsilon = 1$}
\label{fig-2}
\end{center}
\end{figure}

\begin{figure}[!h]
\begin{center}
\includegraphics[width=9cm]{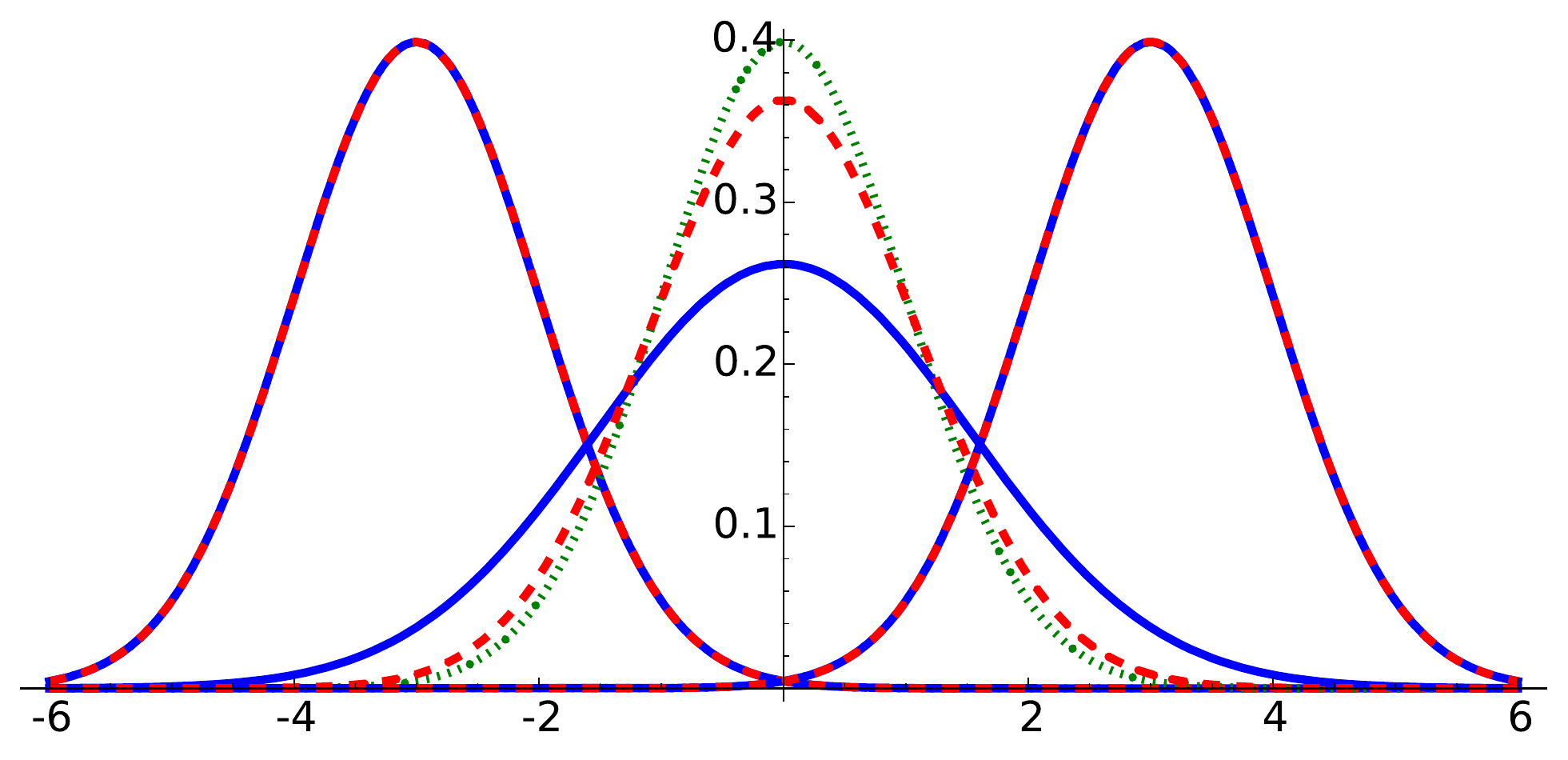}
\caption{Interpolations at time $t=0,1/2,1$, $x_0=-3, x_1=3, \varepsilon = 1$}
\label{fig-3}
\end{center}
\end{figure}

\begin{figure}[!h]
\begin{center}
\includegraphics[width=9cm]{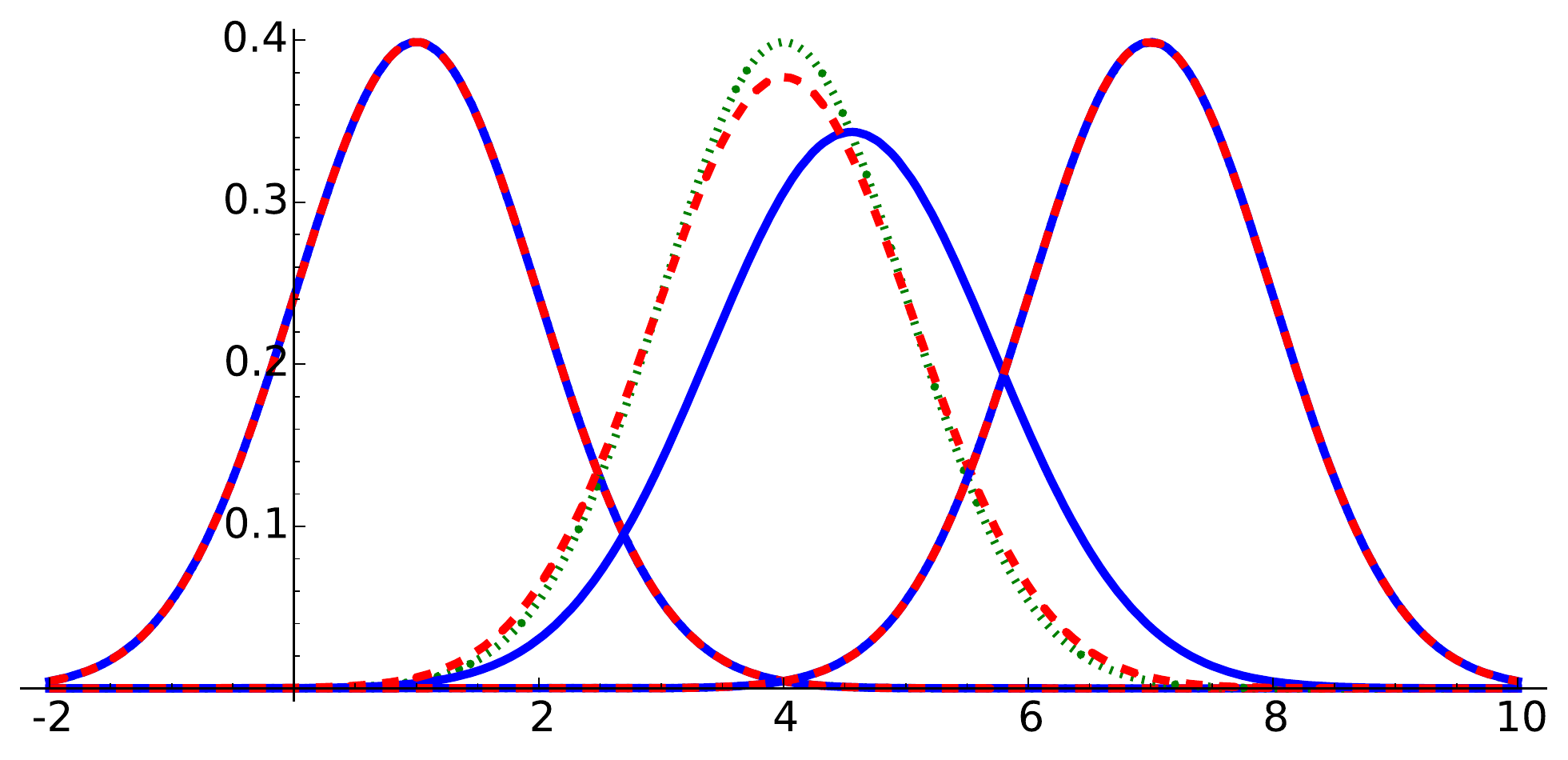}
\caption{Interpolations at time $t=0,1/2,1$, $x_0=1, x_1=7, \varepsilon = 1$}
\label{fig-4}
\end{center}
\end{figure}


\begin{thebibliography}{BBGM12}

\bibitem[AGS08]{ags}
L.~Ambrosio, N.~Gigli, and G.~Savar{\'e}.
\newblock {\em Gradient flows in metric spaces and in the space of probability
  measures}.
\newblock Lectures in Mathematics ETH Z\"urich. Birkh\"auser Verlag, Basel,
  second edition, 2008.

\bibitem[BB00]{benamou-brenier}
J.-D. Benamou and Y.~Brenier.
\newblock A computational fluid mechanics solution to the {M}onge-{K}antorovich
  mass transfer problem.
\newblock {\em Numer. Math.}, 84(3):375--393, 2000.

\bibitem[BBGM12]{bbgm}
D.~Bakry, F.~Bolley, I.~Gentil, and P.~Maheux.
\newblock Weighted {N}ash inequalities.
\newblock {\em Rev. Mat. Iberoam.}, 28(3):879--906, 2012.

\bibitem[BCCNP15]{algo-dauphine}
J.-D. {Benamou}, G.~{Carlier}, M.~{Cuturi}, L.~{Nenna}, and G.~{Peyr\'{e}}.
\newblock {Iterative Bregman projections for regularized transportation
  problems.}
\newblock {\em {SIAM J. Sci. Comput.}}, 37(2):a1111--a1138, 2015.


\bibitem[BGL14]{bgl}
D.~Bakry, I.~Gentil, and M.~Ledoux.
\newblock {\em Analysis and geometry of {M}arkov diffusion operators}, volume
  348 of {\em Grundlehren der Mathematischen Wissenschaften}.
\newblock Springer, Cham, 2014.

\bibitem[BGL15]{bgl2}
D.~Bakry, I.~Gentil, and M.~Ledoux.
\newblock On {H}arnack inequalities and optimal transportation.
\newblock {\em {Ann. Sc. Norm. Super. Pisa, Cl. Sci.}}, 14(3):705--727, 2015.

\bibitem[CGP]{CGP14}
Y.~Chen, T.~Georgiou, and M.~Pavon.
\newblock On the relation between optimal transport and {S}chr{\"o}dinger
  bridges: {A} stochastic control viewpoint.
\newblock Preprint, arXiv:1412.4430.

\bibitem[Kan42]{kantorovich}
L.~Kantorovitch.
\newblock On the translocation of masses.
\newblock {\em C. R. (Doklady) Acad. Sci. URSS (N.S.)}, 37:199--201, 1942.

\bibitem[Kuw15]{kuwada}
K.~Kuwada.
\newblock {Space-time Wasserstein controls and Bakry-Ledoux type gradient
  estimates.}
\newblock {\em {Calc. Var. Partial Differ. Equ.}}, 54(1):127--161, 2015.

\bibitem[L{\'e}o01a]{Leo01a}
C.~L{\'e}onard.
\newblock Convex conjugates of integral functionals.
\newblock {\em Acta Math. Hungar.}, 93(4):253--280, 2001.

\bibitem[L{\'e}o01b]{Leo01c}
C.~L{\'e}onard.
\newblock Minimization of energy functionals applied to some inverse problems.
\newblock {\em Appl. Math. Optim.}, 44(3):273--297, 2001.

\bibitem[L{\'e}o12a]{leonard12}
C.~L{\'e}onard.
\newblock From the {S}chr\"odinger problem to the {M}onge-{K}antorovich
  problem.
\newblock {\em J. Funct. Anal.}, 262(4):1879--1920, 2012.

\bibitem[L{\'e}o12b]{leo12}
C.~L{\'e}onard.
\newblock Girsanov theory under a finite entropy condition.
\newblock In {\em S\'eminaire de {P}robabilit\'es {XLIV}}, volume 2046 of {\em
  Lecture Notes in Math.}, pages 429--465. Springer, Heidelberg, 2012.

\bibitem[L{\'e}o14a]{Leo12b}
C.~L{\'e}onard.
\newblock Some properties of path measures.
\newblock In {\em S\'eminaire de probabilit\'es de Strasbourg, vol. 46.}, pages
  207--230. Lecture Notes in Mathematics 2123. Springer., 2014.

\bibitem[L{\'e}o14b]{leonard14}
C.~L{\'e}onard.
\newblock A survey of the {S}chr\"odinger problem and some of its connections
  with optimal transport.
\newblock {\em Discrete Contin. Dyn. Syst.}, 34(4):1533--1574, 2014.

\bibitem[MT06]{MT06}
T.~Mikami and M.~Thieullen.
\newblock Duality theorem for the stochastic optimal control problem.
\newblock {\em Stoch. Proc. Appl.}, 116:1815--1835, 2006.

\bibitem[Nel67]{nelson}
E.~Nelson.
\newblock {\em Dynamical theories of {B}rownian motion}.
\newblock Princeton University Press, Princeton, N.J., 1967.


\bibitem[{Str}08]{stroock}
D.~W. {Stroock}.
\newblock {\em {Partial differential equations for probabilists.}}
\newblock Cambridge: Cambridge University Press, 2008.

\bibitem[Roy99]{Roy99}
G.~Royer.
\newblock {\em Une initiation aux in{\'e}galit{\'e}s de Sobolev
  logarithmiques}, volume~5 of {\em Cours sp\'ecialis\'es}.
\newblock Soci\'et\'e Math\'ematique de France, 1999.

\bibitem[Vil03]{villani1}
C.~Villani.
\newblock {\em Topics in optimal transportation}, volume~58 of {\em Graduate
  Studies in Mathematics}.
\newblock American Mathematical Society, Providence, RI, 2003.

\bibitem[Vil09]{villani2}
C.~Villani.
\newblock {\em Optimal transport, old and new}, volume 338 of {\em Grundlehren
  der Mathematischen Wissenschaften}.
\newblock Springer-Verlag, Berlin, 2009.

\bibitem[vRS05]{renesse-sturm}
M.-K. von Renesse and K.-T. Sturm.
\newblock Transport inequalities, gradient estimates, entropy, and {R}icci
  curvature.
\newblock {\em Comm. Pure Appl. Math.}, 58(7):923--940, 2005.

\bibitem[Zam86]{Zam86}
J.-C. Zambrini.
\newblock Variational processes and stochastic versions of mechanics.
\newblock {\em J. Math. Phys.}, 27:2307--2330, 1986.

\end{thebibliography}

\end{document}